\documentclass{amsart} 
\usepackage{amssymb,latexsym,textcomp}
\usepackage{bm}
\usepackage{xcolor,graphicx}

\usepackage{cite} 

\usepackage{hyperref}
\hypersetup{colorlinks,allcolors=blue}
\usepackage[OT2,T1]{fontenc}

\title[Isomorphisms of BJ orthogonality]{Isomorphisms of Birkhoff-James orthogonality on finite-dimensional $C^\ast$-algebra}

\author{Bojan Kuzma}
\thanks{This work is supported in part by the Slovenian Research Agency (research program P1-0285 and research projects N1-0210, N1-0296 and J1-50000) and by the Serbian Ministry of Education, Science and Technological
Development through Faculty of Mathematics, University of Belgrade.}
\address{University of Primorska, Glagolja\v{s}ka 8, SI-6000 Koper, Slovenia, and Institute of Mathematics, Physics, and Mechanics, Jadranska 19, SI-1000 Ljubljana, Slovenia}
\email{bojan.kuzma@upr.si}

\author{Srdjan Stefanovi\'c}
\address{University of Belgrade\\ Faculty of Mathematics\\ Student\/ski trg 16-18\\ 11000 Beograd\\ Serbia}
\email{srdjan.stefanovic@matf.bg.ac.rs}



\theoremstyle{plain}
\newtheorem{theorem}{Theorem}[section]
\newtheorem{lemma}[theorem]{Lemma}
\newtheorem{corollary}[theorem]{Corollary}
\newtheorem{proposition}[theorem]{Proposition}

\theoremstyle{definition}
\newtheorem{definition}{Definition}[section]
\newtheorem{example}{Example}[section]

\theoremstyle{remark}
\newtheorem{remark}{Remark}[section]

\DeclareMathOperator{\Ker}{Ker}

\newcommand{\A}{\mathcal A}
\newcommand{\B}{\mathcal B}

\newcommand{\C}{\mathbb C}

\def\CC{\mathbb C}

\def\PP{\mathbb P}
\def\RR{\mathbb R}
\def\Image{\mathop{\mathrm{Im}}}
\def\rank{\mathop{\mathrm{rk}}}
\def\Ker{\mathop{\mathrm{Ker}}}
\def\cV{\mathcal V}
\def\Span{\mathop{\mathrm{span}}}

\numberwithin{equation}{section} 

\begin{document}

\begin{abstract} We classify bijective maps which strongly preserve Birkhoff-James orthogonality on a finite-dimensional complex $C^*$-algebra. It is shown that those maps are close to being real-linear isometries whose structure is also determined.
\end{abstract}

\maketitle

\maketitle

\section{Introduction}
Two vectors in  an inner product  space are orthogonal if their inner product is zero. There are (infinitely) many equivalent ways to state this relation without explicitly using the inner product \cite[part (C)  p.~3]{AlonsoBenitez} and many of them  can be used to extend  orthogonality and define it on general normed spaces. The drawback is that those definitions (which all agree on inner product spaces) typically are no longer equivalent on normed spaces where the norm is not induced by the inner product. Arguably, one of  the best known and useful is the Birkhoff-James (BJ for short) orthogonality, whereby $x$ is orthogonal to $y$ (notation $x\perp y$) if for every scalar $\lambda$ we have
$$\|x+\lambda y\| \ge \|x\|.$$
It has a nice geometrical interpretation in that the origin is the closest point to $x$ on a line passing through $y$.

Presently, we are interested in preservers of this relation. The first results in this direction were obtained by  Koldobsky~\cite{Koldobsky} who classified linear maps which in one direction preserve BJ orthogonality on a real Banach space. This was later extended to normed real or complex spaces by Blanco-Turn\v sek \cite{BlancoTurnsek2006}. The result is in both cases the same: such linear maps are scalar multiples of isometry. We refer to W\'ojcik \cite{Wojcik} for a simple proof as well as further generalization to additive BJ preservers on real space. We were  informed by A. Peralta that, in  a recent preprint~\cite{peralta2024}, they were able to classify additive BJ preservers also on complex normed spaces.

One can try to remove linearity assumption at the expense that  BJ orthogonality is preserved strongly and bijectively. This was done in~\cite{BlancoTurnsek2006} for projective smooth Banach spaces and extended recently by  Ili\v sevi\'c and Turn\v sek \cite{IlisevicTurnsekJMAA2022}, who showed that any BJ isomorphism between two smooth  (possibly noncomplete) normed spaces $X$ and $Y$ is a (conjugate)linear isometry, multiplied by a scalar-valued function.  However, if the norm is nonsmooth the results are different; for example, not every BJ isomorphism on abelian $C^\ast$-algebra  ${\bf c}_0$ (i.e., all sequences converging to $0$) is of this nice form;  see Blanco-Turn\v sek \cite[Example 3.5]{BlancoTurnsek2006} for more. This counterexample works also for finite sequences as long as they contain at least three terms. It is perhaps a little surprising that, at least among finite-dimensional $C^\ast$-algebras, these are the only pathological examples (see our main Theorem~\ref{thm:genera} below). We remark in passing that even without linearity, BJ orthogonality alone  knows  a lot about normed spaces -- it can compute the dimension \cite[Theorems 3.6 and 3.9]{TanakaNonlinear} and characterize  smooth spaces up to (conjugate)linear isometry \cite{TanakaIndagationes,BJNorm} and knows if a given $C^\ast$-algebra is abelian or not \cite[Theorem 3.5]{TanakaAbelian} or if the norm is induced by the inner product~\cite[18.7]{Amir}.

Throughout we will  regard  a finite-dimensional $C^*$-algebra $\A:=\bigoplus_1^\ell M_{n_k}(\CC)$ to be embedded into $M_N(\CC)$ for $N:=n_1+\dots+n_\ell$ (every complex finite-dimensional $C^\ast$-algebra $\A$ has such decomposition, see a book by Goodearl~\cite{goodearl}). Following~\cite{Kuzm-Sush-abelian} we call the sum of $1$-by-$1$ blocks an \emph{abelian summand of $\A$}; this is $\ast$-isomorphic to $(\CC^k,\|\cdot\|_\infty)$, the space of column vectors equipped    with supremum norm, for some $k\ge 0$. The center, $Z(\A)$ equals $\bigoplus_1^\ell \CC I_{n_k}$ and is $\ast$-isomorphic to $(\CC^\ell,\|\cdot\|_\infty)$.  Thus, any positive-definite $P\in Z(\A)$  can be identified with an $\ell$-tuple of positive numbers $(r_1,\dots,r_\ell)$.  Given  $A\in\A$, recall that 
$$M_0(A):=\Ker(\|A\|^2 I- A^\ast A)=\{x\in\CC^N;\;\;\|Ax\|=\|A\|\cdot\|x\|\}$$ is the set of all vectors $x\in\CC^N$ where $A$ attains its norm (i.e.,  such that $\|Ax\|=\|A\|\cdot\|x\|$; here and throughout, a norm without an index will always refer to an Euclidean norm $\|x\|:=\sqrt{x^\ast x}$) see~\cite[Lemma 3.1]{simple} and~\cite{Kuzm-Sush-abelian} for more details. By decomposing $A=A_1\oplus\dots\oplus 
A_{\ell}$ we see that $M_0(A)=M_0(PA)$ for some $P\in Z(\A)$ if and only if $A$ and $PA$ attain their norms at the same summands.

    Let ${\mathbb T}\subseteq\CC$ be the unimodular group. For brevity, we call a bijective map which strongly preserves BJ orthogonality to be a \emph{BJ isomorphism}.

\begin{theorem}\label{thm:genera}
    Let $\Phi\colon\A\to\B$ be a BJ isomorphism between  complex $C^\ast$-algebras $\A$ and $\B$. If  $\A$ is finite-dimensional and has  no abelian summand, then  $\A$ and $\B$ are isomorphic $C^\ast$-algebras and there exists a real-linear surjective isometry $\Psi\colon\A\to\B$, a function $\gamma\colon\A\to{\mathbb T}$ and a central-valued  function $P\colon\A\to Z(\A)$ with $P(X)>0$ (positive-definite) and  $M_0(X)=M_0(P(X)X)$ such that
    $$\Phi(X)=\Psi(\gamma(X)P(X)\cdot X),\qquad \ X\in\A.$$
\end{theorem}
\begin{remark}
The proof will establish even more: By identifying $\B=\A$, for  every  $\Psi\colon\A\to\A$  from Theorem~\ref{thm:genera}, there exists a $C^\ast$ direct sum decomposition $\A=\A_1\oplus\A_2\oplus \A_3$ such that 
    \begin{align*}
        \Psi|_{\A_1}(X_1)&=U_1\sigma(X_1)V_1^\ast,\\\Psi|_{\A_2}(X_2)&=U_2\sigma((X_2^\ast)^T) V_2^\ast=U_2\sigma(\overline{X_2}) V_2^\ast\\
         \Psi|_{\A_3}(X_3)&=U_3\sigma(X_3^\ast) V_3^\ast
    \end{align*} for some unitary  $U_i,V_i\in\A_i$ and permutation~$\sigma$ which permutes the minimal ideals of the same dimension in $\A$, i.e., if $A=\bigoplus_1^\ell A_k\in\A=\bigoplus_1^\ell M_{n_k}(\CC)$, then $\sigma(A):=\bigoplus A_{\sigma(k)}$ and $n_k=n_{\sigma(k)}$. 
\end{remark}
When specializing to a simple finite-dimensional $C^\ast$-algebra the result has a more compact form, interesting in it own:
\begin{theorem}\label{thm:simplecase}
    Let $n\ge 2$ and let $\Phi\colon M_n(\CC)\to M_n(\CC)$ be  a  BJ isomorphism. Then there exists a nonzero scalar-valued function $\gamma\colon M_n(\CC)\to\CC\setminus\{0\}$ such that $X\mapsto \gamma(X)\Phi(X)$ is a (conjugate)linear isometry. 
\end{theorem}
\begin{remark}
    We will also show within the proof of Theorem~\ref{thm:simplecase} that the group of  (conjugate)linear isometries on a $C^\ast$-algebra $M_n(\CC)$, $n\ge 2$, is generated by
    \begin{equation}\label{eq:conjugate-linear_isometries}
        X\mapsto X^\ast\quad\hbox{ and }\quad X\mapsto UXV.
    \end{equation}
    where $(U,V)$ are either both unitary or both conjugate-linear unitary.   
\end{remark}

Actually, we will first prove Theorem~\ref{thm:simplecase}, and then Theorem~\ref{thm:genera} will follow as a corollary, with the help of some of our previous results \cite{kuzma_singla_2025}.

For a general finite-dimensional $C^\ast$-algebra~$\A$  it was shown in~\cite[Corollary~2.9]{kuzma_singla_2025} that a BJ isomorphism  $\Phi\colon\A\to\B$ will map nonabelian summand of $\A$ onto nonabelian summand of $\B$ and will map abelian summand of $\A$ onto abelian summand of $\B$. The restriction to nonabelian summand is covered within Theorem~\ref{thm:genera}; the restriction to abelian summand is covered by Tanaka~\cite[Theorem~5.5]{TanakaNonlinear2} and takes the form 
$$x=(x_1,\dots,x_k)\mapsto(\gamma_1(x)x_{\sigma(1)},\dots, \gamma_k(x)x_{\sigma(k)})$$
for a fixed permutation $\sigma$ and suitable scalars $\gamma_i(x)\in\CC$. One should mention  that~\cite[Theorem~5.5]{TanakaNonlinear2} gives much more: it classifies BJ isomorphisms on a general (possibly nonunital) complex abelian $C^\ast$-algebra; they are weighted composition operators where weights are not constant but depend on the argument and might be discontinuous (consider, e.g., a weight $\omega_f(x):=\sin\left(\frac{\pi}{2x}\right)$ associated to a function $f(x)=x\in {\mathcal C}([0,1])$ and notice that $\omega_f(x)f(x)\in {\mathcal C}([0,1])$). Interestingly, this was proven by showing that BJ isomorphisms on abelian $C^\ast$-algebra are also isomorphisms of a strong BJ orthogonality, defined by~\cite{ArambasicRajic2014}, and then using~\cite[Theorem 5.2]{TanakaContinuous}.

\section{Preliminaries}
For a vector $x$ in a normed space~$\cV$ we let
$$x^\bot=\{y\in\cV;\; x\perp y\}\quad\text{ and }\quad {}^\bot\! x=\{y\in\cV;\; y\perp x\},$$ be its \emph{outgoing} and \emph{incoming neighbourhood}, respectively (throughout, $\perp $ denotes BJ orthogonality). In general they are different; when $\dim\cV\ge 3$, they always coincide if and  only if   the norm is induced by the inner product (see~\cite[18.7]{Amir}), in which case they  further coincide with the usual orthogonal complement. Also, $x$ is  \textit{left-symmetric}  if $x\perp y$ implies $y\perp x$, or equivalently, if  $x^\bot\subseteq {}^\bot\! x$. For $\mathcal S\subseteq \cV$, let
    $$\mathcal L_{\mathcal S}:=\{x\in\mathcal S;\;\;x^\bot	\cap {\mathcal S}\subseteq {}^\bot\!x\cap {\mathcal S}\}.$$
    denote the set of all left-symmetric vectors relative to $\mathcal S$.
        In particular, if $\mathcal S=\cV$, then $\mathcal L_{\cV}$ is the set of all left-symmetric vectors.

It is known that the whole algebra $M_n(\CC)$ contains no left-symmetric points different from zero; see, e.g.,~\cite[Corollary~3.4]{Turnsek2017}.  However, there might exist a subspace  $\cV\subseteq M_n(\CC)$ within which we do find nontrivial relative left-symmetric points; an example is the subspace of all diagonal matrices, which is an abelian $C^\ast$-algebra  and where every rank-one matrix is left-symmetric (see~\cite[Lemma~2.2]{Kuzm-Sush-abelian}).
Within the present section we determine the maximal possible dimension of such subspace~$\cV$, obtained as intersections of common outgoing neighborhoods of smooth elements (they will be called \emph{outgoing spaces}) and also determine all relative left-symmetric points within them. This will be the starting point to prove Theorem~\ref{thm:simplecase}, which will then be used to prove Theorem~\ref{thm:genera}.

Our main tool is the following characterization of BJ orthogonality, valid for finite-dimensional $C^\ast$-algebras. This result was first obtained by Stampfli~\cite{STAMPF} for Hilbert-space operators with $A=I$, the identity; later it was generalized by Magajna~\cite{MAGAJN} for general Hilbert-space operators. A different proof valid for complex or real matrices under the spectral norm was  discovered  by Bhatia-\v Semrl~\cite{Bhat-Semr}, see also~\cite{LiSchneider2002} and~\cite{Benitez2007} for more. Recall that a finite-dimensional $C^\ast$-algebra~$\A=\bigoplus_1^\ell M_{n_i}(\CC)$ is tacitly embedded into $M_N(\CC)$ for  $N=n_1+\dots+n_\ell$; since the definition of BJ orthogonality of elements $A,B\in\A$ requires only two-dimensional space spanned by $A,B$, we have that $A\perp B$ in $\A$ if and only if $A\perp B$ regarded as elements of $M_N(\CC)$.   Denote also 
    $\langle x,y\rangle:=y^\ast x$ for $x,y\in\CC^N$.
        \begin{proposition}[\cite{Kuzm-Sush-abelian}, Lemma 2.1]\label{prop:M_0(A)}
        Let $\A=\bigoplus_{k}^\ell\mathcal M_{n_k}(\CC)$ be a $C^*$-algebra  and  $A,B\in\A$. Then, $B\in A^\bot$ if and only if $\langle Ax,Bx\rangle =0$ for some  normalized  $x\in M_0(A)$. Moreover,  if  $\|A_i\|<\max\{\|A_k\|;\;\;1\leq k\leq\ell\}=\|A\|$, then $$A^\bot = \big(A_1\oplus\dots\oplus A_{i-1}\oplus 0_{n_i}\oplus A_{i+1}\oplus\dots\oplus A_\ell\big)^\bot.$$
    \end{proposition}

We apply this on rank-one $A=xy^\ast$, which attains its norm only on the scalar multiples of  $y$. Hence, $A\perp B$ if and only if $0=\langle Ay, By\rangle=\langle x,By\rangle:=(By)^\ast x=y^\ast B^\ast x$. By conjugating we thus get
\begin{equation}\label{eq:rk-1-perp}
 (xy^\ast)^\bot=\{X;\;\;x^\ast Xy=0\}.   
\end{equation}

   \begin{definition}\label{def}
        We  call $A\in\A$ to be \textit{smooth} if  there does not exist  $B\in\A$ such that $B^\bot\subsetneq A^\bot$.
    \end{definition}
        By~\cite[Lemma 2.5]{Kuzm-Sush-abelian} this definition  agrees with the classical definition of smoothness, i.e., of having a unique supporting functional. Moreover, it was proved in~\cite[lemmata 2.5 and 2.6]{Kuzm-Sush-abelian} that $A=\bigoplus_{1}^\ell A_k\in\A$ is smooth  if and only if
    there exists exactly one index $j$ such that $\|A_j\|=\|A\|$ and $\dim M_0(A_j)=1$. Moreover, if $A$ is smooth and $x=\bigoplus_{1}^\ell x_k\in M_0(A)$, then $x_k=0$ for all $k\neq j$~and \begin{equation}\label{eq:smooth-rank-one}
        A^\bot=\big(0\oplus ((A_jx_j)x_j^*)\oplus 0\big)^\bot.
    \end{equation}
For example, $A=0\oplus E_{st}^j\oplus 0$ are smooth elements for all matrix units $E_{st}^j\in\mathcal M_{n_j}(\CC)$.
\section{Maximal subspaces of relative left-symmetry}

Throughout, $\C^n$ will denote the Euclidean $n$-dimensional complex normed space of column vectors (i.e., $n$-by-$1$ matrices) with $e_1,\dots,e_n$ being its standard basis. Hence, on  $\C^n$, the BJ orthogonality coincides with the usual one.

 A well-known singular value decomposition (SVD for short) says that each complex matrix $A$ can be written as 
 $A=\sigma_1(A) x_1y_1^\ast +\dots+\sigma_n(A) x_n y_n^\ast$, where $x_1,\dots, x_n$ and $y_1,\dots, y_n$ are two orthonormal basis of $\CC^n$ and $\sigma_1(A)\ge \dots\ge \sigma_n(A)\ge 0$ are the singular values of $A$.

 We start with a technical lemma which we will use to prove that some subspaces of $M_n(\CC)$ \emph{cannot have} nontrivial left-symmetricity.

\begin{lemma}\label{lem:lastrow-column} Let $n\ge 2$,  let $x,y\in\C^n$, and let $B=xe_n^\ast +e_ny^\ast$ have nonzero entries at positions $(1n),(n1)$, and $(nn)$. Then its SVD has two summands and  both  have nonzero $(11)$ entry. Moreover, $\sigma_1(B)>\sigma_2(B)$. 
\end{lemma}
\begin{proof}

Clearly,  $B$ is of rank-two so its SVD,
$$B=\sigma_1 u_1v_1^\ast+\sigma_2 u_2v_2^\ast$$
has two summands.

If both $u_1$ and $u_2$ have a zero as a first component, then $B$ has $0$ on position $(1n)$, which is contradictory. Similar conclusion holds for  $v_1$ and $v_2$, for otherwise the $(n1)$.position of  $B$ would vanish. So, at least one among $u_1,u_2$ is not orthogonal to $e_1$ and the same is true for $v_1,v_2$.

Assume, contrary to the claim, that the $(11)$ entry of $\sigma_1 u_1v_1^\ast$ or of $\sigma_2 u_2v_2^\ast$ is zero. Then, due to  $e_1^\ast Be_1=B_{11}=0$, both of them will have $(11)$ entry equal to zero. This is possible only  if  $$e_1^\ast u_1=e_1^\ast v_2=0\quad\hbox{ and }\quad e_1^\ast v_1,e_1^\ast u_2\neq0,$$
 or, vice-versa, 
 $$e_1^\ast u_2=e_1^\ast v_1=0\quad\hbox{ and  }\quad e_1^\ast v_2,e_1^\ast u_1\neq0. $$
 In the former case, notice that $ 
 \sigma_1  u_1(v_1^\ast e_1)=(\sigma_1 u_1v_1^\ast+\sigma_2 u_2v_2^\ast) e_1=Be_1\in\CC e_n\setminus\{0\}$. It follows that $u_1$ is parallel to $e_n$ and by transferring the appropriate scalar from $u_1$ to $v_1$ we achieve that $u_1=e_n$. Thus, comparing zero entries in $e_ny^\ast+xe_n^\ast =\sigma_1 e_nv_1^\ast+\sigma_2 u_2v_2^\ast$ and keeping in mind that $u_2$ is orthogonal to $u_1=e_n$ and $e_1^*u_2\neq0$, we see that $v_2$ is parallel to $e_n$ and again we can assume $v_2=e_n$. But then, $v_1\in v_2^\bot= \CC^{n-1}\oplus 0$ and $u_2\in u_1^\bot=\CC^{n-1}\oplus0$, so $ e_n^\ast Be_n=0$, a contradiction. Thus, the former case is impossible.
The latter case is analogously contradictory. 

Lastly, $\sigma_1(B)=\sigma_2(B)$ if and only if $B$ is a scalar multiple of a partial isometry, or equivalently (by \cite[p.~3]{sakai1971}), $B^\ast B$ is a multiple of a projection. We easily compute
\begin{equation}\label{eq:B*B}
B^\ast B=(e_nx^\ast+ye_n^\ast)(xe_n^\ast+e_ny^\ast)=\|x\|^2e_ne_n^\ast+yy^\ast+(x^\ast e_n)e_ny^\ast+(e_n^\ast x)ye_n^\ast.
\end{equation}
Now, 
$$   B=xe_n^\ast+e_n y^\ast=  
x'e_n^\ast+e_n(y')^\ast$$
 where $x'=x+(y^\ast e_n)e_n$ and $y'=y-(e_n^\ast y)e_n\perp e_n$.  This way we can achieve that in \eqref{eq:B*B}, $y$ is perpendicular to $e_n$. Then, there exists a unitary $W$ with $(W^\ast y,W^\ast e_n)=(\|y\|e_1,e_2)$ and hence, conjugating  \eqref{eq:B*B} with $W$  simplifies into
 $$W^\ast B^\ast BW=\begin{pmatrix}
   \|y\|^2 & \|y\|(e_n^\ast x) \\
   \|y\| (x^\ast e_n)& \|x\|^2
 \end{pmatrix}\oplus 0_{n-2}.$$
 By the assumptions, $e_n^\ast x=B_{nn}\neq0$, so the upper $2$-by-$2$ block of $W^\ast B^\ast BW$ is not a scalar matrix. Consequently, $\sigma_1(B)^2\neq\sigma_2(B)^2$. 
\end{proof}

\begin{lemma}\label{lem:V-(1<rk<n-1)}
   Let $n\ge 2$ and let $v_1,\dots,v_k$ and $u_1,\dots, u_k$ be two collections of  $k\le n-1$ normalized vectors in $\CC^n$. Define
    $$\cV:=\{ X\in M_n(\CC);\;\;v_1^\ast Xu_1=\dots=v_k^\ast Xu_k=0\}.$$ 
If a nonzero $A\in\cV$ is left-symmetric relative to $\cV$, then $\rank A\in\{1,n\}$. 
\end{lemma}
\begin{proof} There is nothing to do if $n=2$, so we assume $n\ge 3$. Suppose otherwise that $2\le \rank A\le n-1$ and let   $$A=\sigma_1 x_1y_1^\ast+\sigma_2 x_2 y_2^\ast+\dots+\sigma_n x_n y_n^\ast$$ be its SVD (so $\sigma_2>0$). Assume there exist a nonzero  $$x\in (v_1^\bot\cap\dots\cap v_k^\bot)\ominus x_2^\bot.$$ 
Then, $B:=xy_2^\ast\in\cV$ and, by Proposition~\ref{prop:M_0(A)} (see also ~\eqref{eq:rk-1-perp}), clearly $A\perp B\not\perp A$. Likewise, if there exists a nonzero $y\in (u_1^\bot\cap\dots\cap u_k^\bot)\ominus y_2^\bot$, then $B:=x_2y^\ast$ has this property. In neither case $A$ is left-symmetric relative to $\cV$. Similarly we argue if $\sigma_3>0$ and there exists a nonzero $x\in (v_1^\bot\cap\dots\cap v_k^\bot)\ominus x_3^\bot$ (then we form $B=xy_3^\ast$) or a nonzero $y\in (u_1^\bot\cap\dots\cap u_k^\bot)\ominus y_3^\bot$ (then we form $B=x_3y^\ast$), etc.

    It remains to consider the case when $$v_1^\bot\cap\dots\cap v_k^\bot \subseteq x_2^\bot\cap\dots\cap x_r^\bot\quad\hbox{ and }\quad u_1^\bot\cap\dots\cap u_k^\bot \subseteq y_2^\bot\cap\dots\cap y_r^\bot$$ with $r:=\rank A$.
    By taking the orthogonal complements on both sides, these are equivalent   to 
   \begin{equation}\label{eq:span_relations}
    \begin{aligned}
        \Span\{x_2,\dots,x_r\}&\subseteq \Span \{v_1,\dots,v_k\}\\
        \Span\{y_2,\dots,y_r\}&\subseteq \Span \{u_1,\dots,u_k\}.  
    \end{aligned}
\end{equation}

We now perform a series of simplifications in the form of transformations $X\mapsto \kappa UXV^\ast$ for unitary $U,V$ and $\kappa>0$ ;  these are scalar multiples of  norm isometries so preserve Birkhoff-James orthogonality. As such, the newly obtained $\kappa UAV^\ast\in U\cV V^\ast$ will remain to be left-symmetric relative to a newly obtained  space $U\cV V^\ast$. For simplicity, and with a slight abuse of notation, we retain the original symbols. 

To start with, by applying transformation $X\mapsto \frac{1}{\sigma_1}X$  we achieve that  $\sigma_1=1$. Next, the two sets on the right of~\eqref{eq:span_relations} are at most $k\le n-1$ dimensional, so there exist unitary $U$ and $V$ such that
\begin{equation}\label{eq:Uv_i}
    \Span\{Vv_1,\dots,Vv_k\},\,\Span\{Uu_1,\dots,Uy_k\}\subseteq
\CC^{n-1}\oplus 0.
\end{equation}
Notice that $v^\ast Xu=(Vv)^\ast (VXU^\ast)(Uu)$. With this in mind, by applying the transformation $X\mapsto VXU^\ast$
we may replace $A\in\cV$ by $VAU^\ast\in V \cV U^\ast$ (and replace $u_i,v_i$ with $Uu_i$ and $Vv_i$, respectively). By a slight abuse of notation, we continue using the original symbols to achieve that \eqref{eq:Uv_i} holds with $U=V=I$.

We may further apply unitaries of the form  $U',V'\in M_{n-1}(\CC)\oplus 1$ to achieve that $U'x_i=V'y_i=e_{i-1}$ for $i=2,\dots,r$. We denote the matrix $U'A(V')^\ast$ again by $A$, the space $U'\cV (V')^\ast$ again by $\cV$ and $U'u_i, V'v_i, U'x_i,V'y_i$ again by $u_i,v_i,x_i,y_i$. Recall that $x_1$ is orthogonal to every member of $\{x_2,\dots, x_r\}=\{e_1,\dots, e_{r-1}\}$, so there further exists a unitary $U$ which fixes $e_1,\dots, e_{r-1},e_n$ and maps $x_1$ into $c_xe_{n-1}+s_x e_n$ for suitable $c_x,s_x\in\CC$ with $|c_x|^2+|s_x|^2=\|x_1\|^2=1$. By transformation $X\mapsto UX$ we achieve that $x_1=c_xe_{n-1}+s_x e_n$;  likewise, by transforming with  $X\mapsto XV^\ast$, we can achieve $y_1=c_y e_{n-1}+s_ye_n$.   This way we achieve that 
$$A=\left(\sum_{i=1}^{r-1}\sigma_{i+1}E_{ii}\right)+
c_x\overline{c_y} E_{(n-1)(n-1)}+c_x\overline{s_y} E_{(n-1)n}+ s_x\overline{c_y} E_{n(n-1)}+s_x\overline{s_y} E_{nn}$$ and that every $B$ of the form $B=xe_n^\ast +e_ny^\ast$ belongs to $\cV$ (because $u_i,v_j\in\CC^{n-1}\oplus 0$).
Therefore, if  $c:=c_x$ and $s:=s_x$ are both nonzero, one defines
\begin{align*} B_{\pm}&=\bar{c}E_{1n}-\bar{s}E_{(n-1)n}\pm\bar{s}E_{n1}+\bar{c}E_{nn}\\
&=\sqrt{1+|c|}\left(\frac{\bar{c}|c|e_1-|c|\bar{s}e_{n-1}+\bar{c}e_n}{\bar{c}\sqrt{2}}\right)\left(\frac{\pm \bar{c}se_1+(|c|^2+|c|)e_n}{\bar{c}\sqrt{2(1+|c|)}}\right)^\ast \\ 
&\hspace{0.5cm}\mbox{}+\sqrt{1-|c|} \left(\frac{-\bar{c}|c|e_1+|c|\bar{s}e_{n-1}+\bar{c}e_n}{\bar{c}\sqrt{2}}\right)\left(\frac{\pm\bar{c}s e_1+(|c|^2-|c|)e_n}{\bar{c}\sqrt{2(1-|c|)}}\right)^\ast \in\cV\end{align*}
and  checks that the right hand side is  a SVD for $B_{\pm}$.
Since $A$ achieves its norm on $y_1=c_y e_{n-1}+s_ye_n$ and maps it into $Ay_1=x_1=ce_{n-1}+s e_n$, we easily calculate
$$\langle Ay_1,B_{\pm}y_1\rangle=\langle x_1,B_{\pm}y_1\rangle=(B_{\pm}y_1)^\ast x_1=0 ,$$
so $A\perp B_{\pm}$. On the other hand, $B_{\pm}$ achieves its norm only on scalar multiples of $b_{\pm}=\frac{\pm\bar{c}s e_1+(|c|^2+|c|)e_n}{\bar{c}\sqrt{2(1+|c|)}}$ and maps it into  $B_{\pm} b_{\pm} = \sqrt{1+|c|}\,\frac{\overline{c}|c|e_1-|c|\overline{s}e_{n-1}+\overline{c}e_n}{\overline{c}\sqrt{2}}$.  Then,
$$\langle B_{\pm} b_{\pm},Ab_{\pm}\rangle=\frac{\overline{c} \overline{s} \left(\pm \sigma _2c+|s|^2
  s_y\right)}{2 | c| }$$
   since $\sigma_2 c\neq0$ we can choose $\pm$ so that the whole expression is nonzero, giving that $B_{\pm}\not\perp A$. So, $A$ is not left-symmetric relative to $\cV$. 

 It only  remains to consider the case when $c_xs_x=0$.  If $c_ys_y\neq0$ we  apply an isometry $X\mapsto X^\ast$ to $A\in\cV$. Then, $$A^\ast= \left(\sum_{i=2}^r\sigma_iE_{(i-1)(i-1)}\right)+y_1x_1^\ast$$  and we   reduced 
 to the just considered case.

 If $c_xs_x=c_ys_y=0$, then $$x_1y_1^\ast\in\CC\{E_{(n-1)(n-1)},E_{(n-1)n},E_{n(n-1)},E_{nn}\}.$$ If now $x_1y_1^\ast$ equals one of the  last three matrices, then $x_1y_1^\ast\in \cV$ and since also $A\in\cV$ we get that $\hat{B}:=A-x_1y_1^\ast =\sigma_2E_{11}+\dots+\sigma_{r}E_{(r-1)(r-1)}\in\cV$ and $A\perp\hat{B}$ but $\hat{B}\not\perp A$.

 Lastly, if $x_1y_1^\ast \in \CC E_{(n-1)(n-1)}$, then $A\in M_{n-1}(\CC)\oplus 0$, it achieves its norm on $e_{n-1}$ and maps it into itself. Here we can choose $B=E_{1n}+E_{n1}+E_{nn}\in\cV$ and by Lemma~\ref{lem:lastrow-column} again conclude  that $A\perp B\not\perp A$.
\end{proof}

\begin{lemma}\label{lem:no-full-rk}
   Let $n\ge 2$ and let $k,u_i,v_j$ and $\cV\subseteq  M_n(\CC)$ be as in Lemma~\ref{lem:V-(1<rk<n-1)}.
   If $A\in\cV$ has $\rank A>k$, then  it is not left-symmetric relative to $\cV$. Specially, there are no full-rank left-symmetric points in $\cV$.
\end{lemma}
\begin{proof}
Let $r:=\rank A\ge k+1$ and let 
$$A= \sigma_1 x_1y_1^\ast+\sigma_2x_2y_2^{\ast}+\dots +\sigma_{r} x_{r} y_{r}^\ast\in\cV$$
be  its SVD, so it attains its norm on $y_1$. We enlarge $\{x_1,\dots,x_r\}$ to an orthonormal basis $\{x_1,\dots,x_n\}$ of $\CC^n$.

The subspace $v_1^{\perp}\cap v_2^{\perp}\cap\dots\cap v_k^{\perp}\subseteq\CC^n$ has codimension at most $k\le n-1$ so it does contain  nonzero vectors. Pick, if possible, any such, say $v$, which is not orthogonal to  $x_i$ for some $2\le i\le r$.  Then $B_i=vy_i^{*}\in\cV$. Further, $A$ attains its norm on $y_1$ and $\langle Ay_1,B_iy_1\rangle=\langle x_1,0\rangle=0$, so $A\perp B_i$.

However, $B_i$ attains its norm only on scalar multiples of  $y_i$  and $\langle B_iy_i,Ay_i\rangle=\langle v,\sigma_i x_i\rangle\neq0$ giving $B_i\not\perp A$. Thus, $A$ is not left-symmetric.

Now, suppose  that every vector in $v_1^{\perp}\cap v_2^{\perp}\cap\dots\cap v_k^{\perp}$ is orthogonal to all $x_i$ ($2\le i\le r$). Then,
$$v_1^{\perp}\cap v_2^{\perp}\cap\dots\cap v_k^{\perp}\subseteq x_2^\bot\cap\dots\cap x_r^\bot=\Span\{x_1,x_{r+1},\dots, x_n\}.$$
The dimension of the space on the left is at least $n-k$. So we must have $n-k\le 1+( n-r)$ or $r\le k+1$. 

In the extremal case, when $r=k+1$, we have equality $$v_1^{\perp}\cap v_2^{\perp}\cap\dots\cap v_k^{\perp}=\Span\{x_1,x_{r+1},\dots,x_{n}\}$$ and then $x_1$ is orthogonal to all $v_i$ ($1\le i\le k$). This implies that $x_1y_1^*\in\cV$, and then $B:=\sigma_2x_2y_2^{\ast}+\dots +\sigma_{r} x_{r} y_{r}^\ast=A-\sigma_1x_1y_1^\ast\in\cV-\cV=\cV$ and again $A\perp B\not\perp A$, so $A$ is not relative left-symmetric. 

Indeed, the only possibility for $A\in\cV$ to be left-symmetric within $\cV$ is that its rank is at most $k\le n-1$. In particular, $A$ cannot be invertible.
\end{proof} 
   \begin{lemma}\label{lem-norank-one-LS}
    Let $n\ge 2$, $k$, $u_i,v_j$ and  $\cV\subseteq M_n(\CC)$ be as in Lemma~\ref{lem:V-(1<rk<n-1)}. Assume, in addition, that $\dim\Span\{v_1,\dots,v_k\}\ge 2$ and $\dim\Span\{u_1,\dots, u_k\}\ge2$.  Then,  $\cV$ contains no rank-one matrix $A$ which would be left-symmetric relative to $\cV$.  
\end{lemma}
\begin{proof} There is nothing to do if $n=2$, so we assume $n\ge 3$. Assume otherwise that a normalized rank-one matrix  $A=xy^\ast\in\cV$ is left-symmetric relative to $\cV$. Without loss of generality, $
    \|x\|=\|y\|=1$. We will now apply a sequence of  isometries of the form $X\mapsto UXV^\ast$ and $X\mapsto X^\ast$ where $U,V$ are unitary,  on an incidence pair $A\in\cV$; the newly obtained  $A$ will clearly remain to be left-symmetric relative to  a newly obtained subspace $\cV$.

       To start, recall that by the assumptions, $u_1,\dots,u_k$ span at least two-dimensional subspace, so we can reorder them to achieve that $u_k$ and $x$ are linearly independent.
       If $x,u_1,\dots,u_k$ spans a proper $r$-dimensional  subspace, we choose unitary $U$ so that  $Ux=e_1$ and $Uu_i\in\CC^{r}\oplus 0_{n-r}$ for each $i$. Otherwise,  $x,u_1,\dots,u_k$ forms a basis,  and then we recursively choose unitaries $U_0,\dots U_k$ so that $U_0x=e_1$ and $U_i$ fixes $e_1\dots, e_i$  (hence also fixes the vectors $u_1,\dots,u_{i-1}$  obtained previously during the process) and maps $u_i$ into  vector $U_iu_i\in\CC^{i+1}\oplus 0_{n-i-1}$ for $i=1,\dots, k$. By the  abuse of notation we still denote the obtained vectors by $u_1,\dots, u_k$. Likewise we recursively change $y,v_1,\dots, v_k$. This way  we achieve that  
    \begin{equation}\label{eq:v_i-inCC^i}
    \begin{aligned}
       A&=E_{11}\\
       u_i&\in\CC^{r_u}\oplus 0_{n-r_u};\qquad &i=1,\dots, k;\qquad  &r_u:=\dim\Span\{e_1,u_1,\dots,,u_k\} \\
       u_i&\in\CC^{i+1}\oplus 0_{n-i-1};\qquad &i=1,\dots, k\qquad &\hbox{if $\Span\{e_1,u_1,\dots, u_k\}=\CC^n$}\\
       v_i&\in\CC^{r_v}\oplus 0_{n-r_v};\qquad &i=1,\dots, k;\qquad  &r_v:=\dim\Span\{e_1,v_1,\dots,v_k\}\\
       v_i&\in\CC^{i+1}\oplus 0_{n-i-1};\qquad &i=1,\dots, k\qquad &\hbox{if $\Span\{e_1,v_1,\dots, v_k\}=\CC^n$}.\\
    \end{aligned}
    \end{equation}  
  Notice also that, due to  our initial reindexation,  $u_k$ and $v_k$ are both linearly independent of $e_1$. 
We now consider three possibilities:

{\bf Case 1.} $v_k\notin \CC e_n$.  Then, since $n\ge 3$, 
there exists a nonzero scalar $\alpha$ and a vector 
\begin{equation}\label{eq:x'}
    x'\perp \{e_1,e_n\}
\end{equation}
 so that $v_k^\ast(\alpha e_1+x'+e_n)=0$ (we can always take $\alpha=1$ except if $v_k\in\CC e_1+\CC e_n$). Due to $u_k\notin\CC e_1$ there also exists a vector $$y'\perp \{e_1,e_n\}$$ with  $u_k^\ast(e_1+y')=0$.
Hence, $$B:=(\alpha e_1+x'+e_n)e_n^\ast+e_n(e_1+y')^\ast$$ belongs to $\cV$, because $u_i^\ast e_n=v_i^\ast e_n=0$ if $i\le k-1$, and clearly $v_k^\ast Bu_k=0$. Also, $A\perp B$ because $A=E_{11}$ achieves its norm on $e_1$ and $e_1^\ast Be_1=0$. However,  by Lemma~\ref{lem:lastrow-column},
$\sigma_1(B)>\sigma_2(B)$ so $B$ achieves its norm on a unique vector $z$ modulo  scalar multiples and, again by  Lemma~\ref{lem:lastrow-column}, $z$ and $Bz$ have nonzero first component. Then, $Az=E_{11}z\in\CC e_1\setminus \{0\}$ cannot be perpendicular to $Bz$, so $B\not\perp A$.

{\bf Case 2.} $v_k\in\CC e_n$, but $u_k\notin\CC e_n$. In this case there exists $y'\perp\{e_1,e_n\}$ and a nonzero $\alpha\in\CC$ so that $u_k^\ast(\alpha e_1+y'+e_n)=0$. Then, $B:=e_1e_n^\ast+ e_n(\alpha e_1+y'+e_n)^\ast$ again belongs to $\cV$ and satisfies $A\perp B\not\perp A$.

{\bf Case 3.} $u_k,v_k\in\CC e_n$. Here we consider the previous two cases for $u_{k-1}$ and $v_{k-1}$ inside $M_{n-1}(\CC)\oplus 0$. If at least one among $u_{k-1},v_{k-1}$ is not parallel with $e_{n-1}$ we can, by following the previous two cases, find $B\in M_{n-1}(\CC)\oplus 0$ of the form $B=(\alpha e_1+x'+e_{n-1})e_{n-1}^\ast+e_{n-1}(e_1+y')^\ast$, respectively,  $B=e_1e_{n-1}^\ast+e_{n-1}(\alpha e_1+y'+e_{n-1})^\ast$ (here, $x'$ satisfies \eqref{eq:x'} and in addition $x'\perp e_{n-1}$, likewise for $y'$) which again belongs to $\cV$ and satisfies $A\perp B\not\perp A$. 

However, if also $u_{k-1},v_{k-1}\in\CC e_{n-1}$, then we simply repeat the procedure inside $M_{n-2}(\CC) \oplus 0_2$. 
Eventually this will stop at the latest  at $M_2(\CC)\oplus 0_{n-2}$. If it stops before, we  always  find $B\in\cV$ with the required properties $A\perp B\not\perp A$. If it stops at $M_2(\CC)\oplus 0_{n-2}$, then  necessarily $k=n-1$ and $(u_i,v_i)=(\CC e_{i+1},\CC e_{i+1})$ for every $i=2,\dots,k$, while $u_1,v_1\in\CC^2\oplus 0_{n-2}$. In this case, we choose a nonzero $b\in\CC^2\oplus 0_{n-2}$, orthogonal to $v_1$, and form $$B=be_2^\ast+e_n(e_1+e_2)^\ast.$$ Clearly, $v_1^\ast B=0$ and $Bu_i\in B(\CC  e_{i+1})=0$ for $i=2,\dots,k$ so $B\in\cV$. Also, $e_1^\ast Be_1=0$ so $A\perp B$. Moreover, $v_1\notin\CC e_1$ so $b\perp v_1$ must have its first entry nonzero. Lastly,  let $V$ be a permutational unitary which swaps $e_2$ and $e_n$; then $BV^\ast =be_n^\ast+e_n(e_1+e_n)^\ast$ has nonzero entries at positions $(1n),(n1),(nn)$ so by  Lemma~\ref{lem:lastrow-column} it achieves its norm at a unique vector $z$ (modulo a scalar multiple); moreover, both $z$ and $BV^\ast z$  have   first component nonzero. Then, $BV^\ast\not\perp A$, so also $B=(BV^\ast)V\not\perp AV=A$. Indeed, $A$ is not left-symmetric relative to $\cV$. 
\end{proof}

 We next give the first examples of a maximal-dimensional outgoing space that does contain relative left-symmetricity. Notice that a matrix $X$ has zero at position $(1i)$ if and only if $e_1^\ast Xe_i=0$; thus the spaces $\cV_p$ which appear in Lemma below are specializations of spaces  from Lemma~\ref{lem:V-(1<rk<n-1)}.
\begin{lemma}\label{lem:zadnja}
    Let $n\ge 2$ and let $\cV_p\subseteq M_n(\CC)$ be a subspace of all matrices  with zeros at positions $(1p),\dots,(1n)$. If $p\ge 3$, then $\cV$ contains no  nonzero left-symmetric points. If $p=2$, then there exist nonzero left-symmetric points in $\cV$ and all are of rank-one. In addition, if $n\ge 3$, then $A\in\cV_p$ is left-symmetric if and only if 
    $A\in\CC E_{11}$.     
\end{lemma}
\begin{remark}\label{rem:} 
If $n=2$, then, with additional arguments, one can show that the set of relative left-symmetric points  in $\cV_2=\left(\begin{smallmatrix}
        \CC &0\\
        \CC&\CC
\end{smallmatrix}\right)\subseteq M_2(\CC)$ equals $\CC E_{11}\cup \CC E_{21}\cup \CC E_{22}$.
\end{remark}
  \begin{proof}[Proof of Lemma \ref{lem:zadnja}]
 Let a nonzero $A\in\cV_p$, with SVD 
  $$A=\sigma_1 x_1y_1^\ast+\sigma_2 x_2y_2^\ast+\dots +\sigma_n x_ny_n^\ast;\qquad \sigma_1\ge\dots\ge \sigma_n\ge0$$ be left-symmetric. Clearly $A$ attains its norm on $y_1$ and maps it into $Ay_1=\sigma_1 x_1$.
  
  Assume first  $\rank A\ge 2$.  If  $x_2'$, the projection of $x_2$ to $\mathrm{span}\{e_2,\dots,e_n\}$ is nonzero, then $B:=x_2'y_2^\ast\in\cV_p$ satisfies $A\perp B\not\perp A$, a contradiction. If $x_2'=0$, i.e., if $x_2\in\CC e_1$, then $x_1,x_3,\dots, x_n$ are orthogonal to $e_1$ and hence $\sum_{i\neq2}\sigma_i x_iy_i^\ast\in\cV_p$. Then also $B:=x_2y_2^\ast=A-\sum_{i\neq2}\sigma_i x_iy_i^\ast\in\cV_p-\cV_p=\cV_p$ which again satisfies $A\perp B\not\perp A$, a contradiction. 
  
  Thus, $\rank A= 1$ and $A=\sigma_1x_1y_1^\ast \in\cV_p$.
  We proceed under the assumption that $n\ge 3$. Decompose
  $$x_1=\alpha e_1+x_1';\qquad x_1'\perp e_1$$ 
  and assume next that $x_1'$ is nonzero. 

  Then, the compression $A'=\sigma_1 x_1'y_1^\ast$ belongs 
  to a subspace $\cV_p':=\{X;\;\;e_1^\ast X=0\}\subseteq\cV_p$ consisting of matrices with vanishing first row. Also, every  isometry of the form $X\mapsto X V^\ast$ where $V$ is unitary, leaves $\cV_p'$ invariant. We can choose $V$ so that $A'V^\ast\in\cV_p''=\{X;\;\;e_1^\ast X=Xe_1=0\}$, which is isometric to $M_{n-1}(\CC)$, and consequently such contains no left-symmetric points (see~\cite[Corollary~3.4]{Turnsek2017}). Therefore, there exists $BV\in\cV_p''$ with $A'V\perp BV\not\perp A'V$. Notice that the image, $\Image B$, is orthogonal to $e_1$ and notice that $A$ and $A'$ attain their norm on the same vector $y_1$. Then, $A'\perp B\not\perp A'$ implies $A\perp B\not\perp A$, a contradiction.

  Thus, $x_1'=0$ and $x_1=\alpha e_1$. If now $p\ge 3$, then $A=\sigma_1(\alpha e_1)y_1^\ast $ belongs to a subspace $M_{p-1}(\CC)\oplus0\subseteq\cV_p$, which contains no left-symmetric points (see again~\cite[Corrolary~3.4]{Turnsek2017}).

  Assume  finally $p=2$.
     Clearly, $A=E_{11}\in\cV_p$ achieves its norm only on $e_1$, so if $A\perp B$ for some $B\in\cV_p$, then $0=\langle Ae_1,Be_1\rangle=\langle e_1, Be_1\rangle$. Together with $B\in\cV_p$ this implies that $B$ vanishes in the first row. So, $\Image B$ is orthogonal to $\CC e_1=\Image A$ and hence $B\perp A$. This shows that $E_{11}$ is left-symmetric, hence, so is $\CC  E_{11}$, by homogeneouity of BJ orthogonality.
\end{proof}

\begin{corollary}
    Let $n\ge 2$ and let $A_1,\dots,A_k\in M_n(\CC)$ be smooth points. If $\cV:=A_1^\bot\cap\dots\cap A_k^\bot$ contains a nonzero relative left-symmetric point, then $k\ge n-1$. If $k=n-1$, then  there exist $n-1$ linearly independent vectors $v_1,\dots, v_{n-1}$ and a nonzero vector $u$ such that 
    $$A_i^\bot= (uv_i^\ast)^\bot;\qquad i=1,\dots, n-1$$
    or 
    $$A_i^\bot= (v_iu^\ast)^\bot;\qquad i=1,\dots, n-1.$$ 
    Also,  in this case, if $n\ge 3$, and $v_n\in\Span\{v_1,\dots,v_{n-1}\}^\bot$, then the set of relative left-symmetric points in $\cV$ equals $\CC A$ for  $A=uv_n^\ast\in\cV$ or $A=v_nu^\ast\in\cV$, respectively. Same holds if $n=2$, but we have additional left-symmetric matrices, all of rank-one.
\end{corollary}
\begin{proof}
By \eqref{eq:rk-1-perp} and \eqref{eq:smooth-rank-one} there exist rank-one $u_iv_i^\ast$ with $$A_i^\bot=(u_iv_i^\ast)^\bot=\{X;\;\;u_i^\ast Xv_i=0\}.$$
It implies that $\cV$ is the space from Lemma~\ref{lem:V-(1<rk<n-1)}. Assume $k\le n-1$. By lemmata~\ref{lem:V-(1<rk<n-1)}---\ref{lem:no-full-rk}, $\cV$ contains no relative left-symmetric matrices of rank bigger than one. By Lemma~\ref{lem-norank-one-LS} it also contains no rank-one unless one of $\Span\{u_1,\dots,u_k\}$, $\Span\{v_1,\dots,v_k\}$ is one-dimensional. 

By applying, if needed, an isometry  $X\mapsto X^\ast$ on an incidence pair $A\in\cV$ we can achieve that $\dim\Span\{u_1,\dots,u_k\}=1$. By further applying unitary similarity $X\mapsto VXU^\ast$ on $A\in\cV$ and transferring appropriate scalars to $v_i$, we can assume $u_1=\dots=u_k=e_1$ and $\Span\{v_1,\dots,v_k\}=0_{p-1}\oplus\CC^{n-p-1}$ for some $p\le k$. Notice from~\eqref{eq:rk-1-perp} that then
$$\cV=\{X;\;\;e_1^\ast Xv_1=\dots=e_1^\ast Xv_k=0\}=\{X;\;\;e_1^\ast Xe_{p}=\dots=e_1^\ast Xe_n=0\}.$$
However, such $\cV$ is a space~$\cV_p$ from  Lemma~\ref{lem:zadnja}, by which a nonzero $A\in\cV_p$ is relative left-symmetric if and only if $p\le 2$. This implies that $k=n-1$. By the same lemma, if $n\geq3$, the only relative left-symmetric points in $\cV_p$ are $\CC E_{11}$.
\end{proof}
We conclude with a general lemma that describes the linear dependence of two matrices using BJ orthogonality.  Define
$${\mathcal R}_1=\{A\in M_n(\CC);\;\; \rank\,A=1\}.$$
\begin{lemma}\label{lem:A=B}
    Let $n\ge 2$ and let $A, B\in M_n(\CC)$.
    Then the following two statements are equivalent. 
    \begin{itemize}
        \item[(i)] $\CC A=\CC B$.
        \item[(ii)] $({}^\bot\! A)\cap\mathcal{R}_1=({}^\bot\! B)\cap\mathcal{R}_1$.
    \end{itemize}
\end{lemma}
\begin{proof}
    We only prove the nontrivial part. Notice that rank-one $R=xy^\ast $ belongs to ${}^\bot X$ if and only if  $R\perp X$, and since $R$ attains its norm only on scalar multiples of $y$, this is further equivalent to $x^\ast Xy=0$ (c.f.~\eqref{eq:rk-1-perp}). 
    
    The condition (ii) is, hence,  equivalent to $x^\ast Ay=0\Longleftrightarrow x^\ast By=0$ for every $x,y\in\CC^n$. It is easily seen that this is further  equivalent to the fact that $Ay$ and $By$ are linearly dependent vectors for every $y$, that is, $A,B$ are locally dependent matrices. Now, if  $\rank(A)\ge 2$ or if $\rank(B)\ge 2$, then  it is well known, and easy to see that  (i) holds (c.f.~\cite[Lemma~2.2]{Hou1989}).
    
    If $A=uv^\ast$ and $B=zw^\ast$ are both of rank-one, then, there exist $x$ not perpendicular to $v$ nor to $w$. So, $Ax$ and $Bx$ are both nonzero. Being locally dependent, implies, after transferring the appropriate scalars to the second factor, that $u=z$. 
    If $v,w$ are linearly independent, there would exist $y$ with $v^\ast y\neq0$ but $w^\ast y=0$, and hence $uy^\ast$ would  be rank-one in ${}^\bot\! A$ but not in ${}^\bot\! B$, a contradiction. So, $u=z$ and $\CC v=\CC w$.
    
    Finally, $A=0$ if and only if ${}^\bot\!A\cap{\mathcal R}_1={\mathcal R}_1$, so (ii) implies $B=0$.
\end{proof}

\section{BJ isomorphisms on \texorpdfstring{$M_n(\CC),n\geq3$}{TEXT}}
Within this and the next section, we will classify BJ isomorphisms $\Phi\colon M_n(\CC)\to M_n(\CC)$.
\begin{lemma}\label{lem:0classifion}
We have $A=0$ if and only if  $A\perp A$.
\end{lemma}
\begin{proof}
Trivial verification.    
\end{proof}
It follows easily that every (possibly nonbijective) map between two normed spaces, which strongly preserves BJ orthogonality, maps $0$ into $0$ and maps nonzero elements to nonzero elements. We will constantly rely on this fact.

We now characterize rank-one matrices in terms of BJ orthogonality alone:
\begin{lemma}\label{lem:rk-1classificatin} Let $n\ge 2$. The following are equivalent for a nonzero $A\in M_n(\CC)$.
\begin{itemize}
\item[(\itshape i)] $\rank A=1$.
\item[(\itshape ii)] There exist $n-1$ smooth elements $A_2,\dots, A_n$ such that $A$ is left-symmetric relative to $\bigcap\limits_{i=2}^n A_i^\bot$.
\end{itemize}
\end{lemma}
\begin{proof}
    ({\itshape i}) $\Longrightarrow$ ({\itshape ii}). Write $A=xy^\ast$, enlarge $x_1:=x$ with $x_2,\dots,x_n$ to orthogonal basis of $\CC^n$ and define $A_i:=x_iy^\ast,2\leq i\leq n$. Note that $A_2,\dots,A_{n}$ are all smooth (see text after Definition \ref{def}). By applying BJ isomorphism, induced by a linear isometry $X\mapsto \kappa UXV^\ast$ for  suitably chosen unitaries $U,V$ and scalar $\kappa>0$ we can assume that $A=E_{11}$ and  $A_i=E_{1i}$ for all $2\le i\le n$. By \eqref{eq:rk-1-perp}, $A_i^\bot=\{X;\;\;e_1^\ast Xe_i=0\}$. Then,  $A=E_{11}\in\bigcap\limits_{i=2}^n E_{1i}^\bot=\bigcap\limits_{i=2}^n A_i^\bot=\{X;\;\;e_1^\ast Xe_2=\dots=e_1^\ast Xe_n=0\}$ and, by Lemma~\ref{lem:zadnja} (and its Remark~\ref{rem:} if $n=2$), it is left-symmetric relative to $\bigcap\limits_{i=2}^n A_i^\bot$. 

    ({\itshape ii}) $\Longrightarrow$ ({\itshape i}) By  \eqref{eq:smooth-rank-one}, for every smooth element $X$, there exists rank-one $R=xy^\ast$ such that $R^\bot=X^\bot$. Hence, we can assume that every smooth $A_i$ is of  rank-one and write it as  $$A_i=u_iv_i^{\ast}.$$ From lemmata~\ref{lem:V-(1<rk<n-1)}--\ref{lem-norank-one-LS} we know that there is no nonzero left-symmetric point inside $\bigcap\limits_{i=2}^n A_i^\bot$ if both $\Span\{u_2,\dots,u_n\}$ and $\Span\{v_2,\dots,v_n\}$ are at least two-dimensional. By applying BJ isomorphism induced by conjugate linear isometry $X\mapsto X^{\ast}$ and absorbing a appropriate scalars to $u_i$'s, we can assume that $v_i=v$ for all  $2\le i\le n$. By applying a linear isometry $X\mapsto UXV^\ast$, we can achieve $v=e_1$ and $\Span\{u_2,\dots,u_n\}=\Span\{e_p,\dots,e_n\}$ for some $p\ge 2$ (actually, $p=n-\dim\Span\{u_2,\dots,u_n\}-1$). It is easy to see by \eqref{eq:rk-1-perp} that then $\bigcap\limits_{i=2}^n A_i^\bot$ coincides with the space of all matrices which vanish at positions $(1p),\dots,(1n)$. Since we assumed that $A$ is left-symmetric in this intersection, by Lemma~\ref{lem:zadnja} (and its Remark~\ref{rem:} if $n=2$) $\rank A=1$ (and $p=2$).
\end{proof}

We apply the previous two lemmata on bijective maps which strongly preserve BJ orthogonality.  Recall that
$${\mathcal R}_1=\{A\in M_n(\CC);\;\; \rank\,A=1\}$$
and recall  that matrices $A,B$ are adjacent if $A-B\in {\mathcal R}_1$, that is, if $\rank(A-B)=1$. Notice that if $A=xy^\ast$ and $B=uv^\ast$ are both of rank-one, then $\rank(A-B)\le1$ if and only if $x\in \CC u$ or $y\in\CC v$ (or both).
\begin{corollary}\label{cor:Phi-preserves-adjacency} Let $n\ge 2$.
    If  $\Phi\colon M_n(\CC)\to M_n(\CC)$ is a BJ isomorphism, then $\Phi({\mathcal R}_1)={\mathcal R}_1$. Moreover, $\Phi|_{{\mathcal R}_1}$ maps linearly dependent pairs to linearly dependent pairs and, if $n\ge 3$, strongly preserves adjacency. 
\end{corollary}
\begin{proof}
    By Definition \ref{def}, every BJ isomorphism $\Phi$ maps the set of smooth elements bijectively onto itself. Also, $\Phi(A^\bot)=\Phi(A)^\bot$, and, moreover, the implication  $A\perp B\Rightarrow B\perp A$ holds for every $B\in\cV$, if and only if  also $\Phi(A)\perp Y\Rightarrow Y\perp \Phi(A)$ holds for every $Y\in\Phi(\cV)$ (i.e., $\Phi$ maps relative left-symmetric points in $\cV$ to relative left-symmetric points in $\Phi(\cV)$). Thus, as an immediate corollary of lemmata~\ref{lem:0classifion}--\ref{lem:rk-1classificatin}, $\Phi$ strongly preservers rank-one matrices.

    It remains to prove that the restriction $\Phi|_{{\mathcal R}_1}$ preserves linear dependence and, if $n\ge 3$, adjacency.  To see this, choose adjacent rank-one $R_1=x_1y_1^\ast$ and $R_2=x_2y_2^\ast$. Then, either $\CC x_1=\CC x_2$  or $\CC y_1=\CC y_2$. By  
    applying, if needed, the BJ isomorphism induced by isometry $X\mapsto X^\ast$ (and absorbing the appropriate scalars to the other side of the product) we can achieve that 
    $$x_1=x_2=:x.$$  Now, if $R_1=\alpha R_2$ for some nonzero scalar~$\alpha$, then $R_1^\bot=R_2^\bot$, so also $\Phi(R_1)^\bot=\Phi(R_2)^\bot$. In view of \eqref{eq:rk-1-perp} this implies that rank-one $\Phi(R_1),\Phi(R_2)$ are linearly dependent. 
    
   Otherwise,  $\CC y_1\neq \CC y_2$. Here, we assume $n\ge 3$ extend $y_1,y_2$ with $y_3,\dots, y_n$ to a basis of $\CC^n$  and form  $n-2$ additional pairwise adjacent smooth elements $R_i=xy_i^{\ast}$.  By applying suitable isometry $X\mapsto UXV^\ast$, we can achieve that $x=e_1$ and $$\Span\{y_1,y_2,\dots ,y_{n-1}\}=\Span\{e_2,e_3,\dots, e_n\}.$$  By~\eqref{eq:rk-1-perp} we then have 
   $$\bigcap\limits_{i=1}^{n-1} R_i^\bot=\{X;\;\;e_1^\ast Xe_2=\dots=e_1^\ast Xe_n=0\}$$
   and, by Lemma~\ref{lem:zadnja}, this space contains a nonzero left-symmetric point, so the same must hold for $\Phi\left(\bigcap\limits_{i=1}^{n-1} R_i^\bot\right)=\bigcap\limits_{i=1}^{n-1} \Phi(R_i)^\bot$. Moreover, since $R_i$ are rank-one, then  $\Phi(R_i)=u_iv_i^\ast$ are also rank-one by the first part of the proof. Hence, by lemmata \ref{lem:V-(1<rk<n-1)}---\ref{lem-norank-one-LS}, either $\dim\Span\{u_1,\dots, u_{n-1}\}=1$ or $\dim\Span\{v_1,\dots, v_{n-1}\}=1$. In either case, $\Phi(R_1),\dots,\Phi(R_{n-1})$ must be pairwise adjacent. Since we assumed $n\ge 3$, then in particular, $\Phi(R_1),
   \Phi(R_2)$ are adjacent. By applying the same arguments to $\Phi^{-1}$ we see that adjacency is strongly preserved.
\end{proof}

From here on, the arguments for the case $n \geq 3 $ are standard (see, e.g.,~\cite{Westwick1967, Hou1989}); we provide them in full for the sake of completeness. We assume throughout the remainder of this section that $n \geq 3$ and defer the case $n = 2$ (where we still need to verify that adjacency is preserved) to a separate section. Given a subset \( Y \subseteq \mathbb{C}^n \), let us define:
$$\hbox{a row: }x\otimes Y:=\{xy^\ast;\;\;y\in Y\}\quad\hbox{and}\quad \hbox{ a column: } Y\otimes x:=\{yx^\ast;\;\;y \in Y\}.$$

\begin{lemma}\label{lem:row-ro-column}
    For each nonzero $x$ there exists a nonzero $y$ such that $\Phi(x\otimes \CC^n)=y\otimes\CC^n$ or $\Phi(x\otimes \CC^n)=\CC^n\otimes  y$.
\end{lemma}
\begin{proof}
Notice that $x\otimes \CC^n$ and $\CC^n\otimes x $ are maximal (in set-theoretical sense) subsets of $ M_n(\CC)$ consisting of pairwise adjacent rank-one matrices, and every  maximal subset of pairwise adjacent rank-one matrices is of this form. From here, the result follows easily from Corollary~\ref{cor:Phi-preserves-adjacency}.
\end{proof}
If needed, we can replace $\Phi$ with a BJ isomorphism $X\mapsto \Phi(X)^\ast$ to achieve that 
\begin{equation}\label{eq:phi(e_1-otimesCC^n}
    \Phi(e_1\otimes\CC^n)=y\otimes\CC^n 
\end{equation} for some nonzero $y$. In the sequel, we will always assume \eqref{eq:phi(e_1-otimesCC^n}.

\begin{lemma}\label{lem:Phi-preserves-rows}
    For each nonzero $x\in\CC^n$ there exists a nonzero $y\in\CC^n$ such that $\Phi(x\otimes \CC^n)=y\otimes\CC^n$.
\end{lemma}

\begin{proof}
    We have two cases:
    
 (i) $x\in\CC e_1$. Then $x\otimes\CC^n=e_1\otimes\CC^n$ and the conclusion of the lemma is a consequence of the normalization \eqref{eq:phi(e_1-otimesCC^n}.

(ii) $x\notin\CC e_1$. Then $(x\otimes\CC^n)\cap(e_1\otimes\CC^n)=0$ so $\Phi((x\otimes\CC^n)\cap(e_1\otimes\CC^n))=0$. By  Lemma~\ref{lem:row-ro-column}, we know that $\Phi(x\otimes\CC^n)$ is either a row or a column. It cannot  be  a column, because every row and column have a nonzero intersection.
\end{proof}

\begin{lemma}\label{lem:Phi-presevers-columns}
    For each nonzero $x\in\CC^n$ there exists a nonzero $y\in\CC^n$ with $\Phi(\CC^n\otimes x)=\CC^n\otimes y$. In particular, $\Phi$ maps rows onto rows and columns onto columns.
\end{lemma}
\begin{proof}
A column $\CC^n\otimes x$ intersects every row in a rank-one matrix, while two rows intersect in a rank-one matrix if and only if they are equal. Hence, if $\Phi(\CC^n\otimes x)=y\otimes \CC^n$ would be a row, then, by Lemma~\ref{lem:Phi-preserves-rows},   $\Phi|_{\mathcal{R}_1}$ would have to map every row to $y\otimes \CC^n$, contradicting its surjectivity.
\end{proof}
Let $\PP(\CC^n):=\{[x]:=\CC x;\;\;x\in\CC^n\setminus\{0\}\}$ be a projective space. We call the one-dimensional space $[x]$ a line for short (or also a projective point). Two lines, $[x]$ and $[y]$, are perpendicular, denoted by $[x]\perp[y]$, if any of their representatives are perpendicular, i.e, if $y^\ast x=0$.

 Notice that  $[x_1]=[x_2]$ if and only if  $x_1\otimes\CC^n=x_2\otimes\CC^n$. Therefore, by lemmata~\ref{lem:Phi-preserves-rows}---\ref{lem:Phi-presevers-columns}, $\Phi$ induces a well-defined bijection $\phi\colon\PP(\CC^n)\to\PP(\CC^n)$, given by
$$\phi([x])=[y]\quad\hbox{if}\quad\Phi(x\otimes\CC^n)=y\otimes \CC^n.$$
\begin{lemma}
    $\phi$ strongly preserves perpendicularity of lines.
\end{lemma}
\begin{proof}
By~\eqref{eq:rk-1-perp}, $[x]\perp [y]$ if and only if every     $R\in x\otimes \CC^n $ and every $S\in y\otimes\CC^n$ are BJ orthogonal.
\end{proof}
It now follows  from Faure's paper,~\cite[Corollary 4.5 and Lemma~4.2.]{Faure2002} that $\phi[x]=[Ux]$ for some (conjugate)-linear isometry $U$. (We caution here that if $U$ is conjugate-linear, then its adjoint is also conjugate-linear and defined by $\langle Ux,y\rangle=\langle U^\ast y,x\rangle$; we still have $U^\ast U=I$, however, now, $\langle Ux,Uy\rangle=\langle y,x\rangle$.) As such, 
$$\Phi(x\otimes \CC^n)=(Ux)\otimes\CC^n;\qquad x\in\CC^n.$$

We next temporarily form a  BJ isomorphism $\Psi\colon X\mapsto \Phi(X^\ast )^\ast$. By  Lemma~\ref{lem:Phi-presevers-columns}, it maps a row $x\otimes \CC^n$ onto $\Psi(x\otimes\CC^n)=\Phi(\CC^n\otimes x)^\ast=(\CC^n\otimes z)^\ast=z\otimes \CC^n$ for some $z=z_x$, i.e., into another row. We can, hence, repeat the above arguments on $\Psi$ to see that there exists a (conjugate)linear isometry $V$ such that $\Psi(x\otimes\CC^n)=(Vx)\otimes\CC^n$. Equivalently, $\Phi(\CC^n\otimes y)=\CC^n\otimes (Vy)$.  Thus,
\begin{align*}
\Phi(xy^\ast)\in\Phi(x\otimes\CC^n\cap \CC^n\otimes y)&=\Phi(x\otimes \CC^n)\cap\Phi(\CC^n\otimes y)\\
&=(Ux)\otimes\CC^n\cap \CC^n\otimes (Vy)=\CC (Ux)(Vy)^\ast.
\end{align*}
It implies that
$$\Phi(xy^\ast)=\gamma_{xy^\ast} (Ux)(Vy)^\ast;\qquad \gamma_{xy^\ast}\in\CC\setminus\{0\}. $$

Assume $U$ is conjugate-linear, while  $V$ is linear. Let 
$$\hat{J}\colon x=\sum x_i e_i\mapsto\bar{x}:= \sum \bar{x}_i e_i$$ be the conjugation relative to the standard basis. Then, $U=(U\hat{J})\hat{J}$ where $(U\hat{J})$ is unitary. As such, the map $\hat{\Phi}\colon X\mapsto (U\hat{J})^\ast\Phi(X)V$ is again a BJ isomorphism and its restriction to rank-one matrices satisfies 
$$\hat{\Phi}(xy^\ast)\dot{=}\bar{x}y^\ast$$
(for simplicity, we have denoted and will continue to denote the fact that $\mathbb{C} A = \mathbb{C} B$, i.e., that $A$ and $B$ are equal up to a scalar, by $A \dot{=} B$).
In particular, $\hat{\Phi}$  fixes all $E_{ij}$ (modulo scalars). Then,  
$\cV:=\bigcap_{(ij)\notin\{(11),(12),(21),(22)\}}E_{ij}^\bot = M_2(\CC)\oplus 0_{n-2}=M_2(\CC)$ is invariant for $\hat{\Phi}$. Conversely, if $\hat{\Phi}(B)\in\cV$, then $\hat{\Phi}(E_{ij})\dot{=}E_{ij}\perp\hat{\Phi}(B)$, so $E_{ij}\perp B$ for $i,j\ge 3$ giving $B\in \cV$. Hence, $\Phi|_{\cV}\colon\cV\to\cV$ is a BJ isomorphism.  We will treat such maps separately in a self-inclusive section devoted to  the case  $n=2$.  The result (see Lemma~\ref{lem:phi=psi} below) is that $\Phi|_{\cV}( xy^\ast)\dot{=}\phi(x)\phi(y)^\ast$
for some bijection $\phi\colon\CC^2\to\CC^2$. This clearly contradicts $\hat{\Phi}(xy^\ast)\dot{=}\bar{x}y^\ast$. 

 Hence, if $U$ is conjugate-linear, then $V$ is also; likewise  one can show the converse. As such, $(Ux)(Vy)^\ast=U(xy^\ast)V^\ast$.

Notice that $X\mapsto U^\ast XV$ is a (conjugate)linear isometry on $M_n(\CC)$ hence it induces a BJ isomorphism. As such, we  can replace $\Phi$ by a map $X\mapsto U^\ast 
\Phi(X)V$ to achieve that rank-one matrices are fixed, modulo scalar multiples. 
\begin{equation}\label{eq:modifiedPhi}
\Phi(xy^\ast)=\gamma_{xy^\ast}xy^\ast;\qquad xy^\ast\in{\mathcal R}_1.
\end{equation}

\begin{proof}[Proof of Theorem~\ref{thm:simplecase} for $n\ge 3$]
We can find (conjugate)linear isometries $U,V$ such that the map $X\mapsto U^\ast\Phi(X)V$, which we still denote by $\Phi$, is a BJ isomorphism and satisfies~\eqref{eq:modifiedPhi} for rank-one $A=xy^\ast$. By Lemma~\ref{lem:0classifion}, $\Phi(0)=0$, so  \eqref{eq:modifiedPhi} holds also for $A=0$. Finally, choose any matrix $A$ with $\rank A\ge 2$ and  in Lemma~\ref{lem:A=B} insert $B:=\Phi(A)$ to get that $B=\gamma_A A$. 
\end{proof}
\section{BJ isomorphisms on \texorpdfstring{$M_2(\CC)$}{TEXT}}
We continue with the proof of Theorem~\ref{thm:simplecase} and now consider BJ isomorphisms $\Phi$ on $M_2(\CC)$. This requires separate treatment since fundamental theorem of  projective geometry is no longer available. We will make  a series of modifications on $\Phi$, whereby we compose $\Phi$ with (conjugate)linear isometries of the form $X\mapsto X^\ast$ or $X\mapsto UXV^\ast$; the newly obtained maps, which we will, with a slight abuse of notation, continue to denote by $\Phi$, will remain  BJ isomorphism, but will fix an increasingly large sets. Our modifications are based on a series of lemmata.
To start with, by Corollary~\ref{cor:Phi-preserves-adjacency} we know that $\Phi$ maps rank-one onto rank-one (however we have not yet shown that $\Phi$ preserves adjacency). By modifying it with $U\Phi(\cdot)V^\ast$ for suitable unitaries $U,V$  we can achieve that 
\begin{equation}\label{eq:E11fixed}
    \Phi(E_{11})\dot{=}E_{11}
\end{equation}

\begin{lemma}\label{lem:n=2-E_(11)}
    The following are equivalent for a rank-one $B\in M_2(\CC)$.
    \begin{itemize}
        \item[(i)] $E_{11}\perp B$ and $E_{11}^\bot\cap B^\bot$ consists of matrices with rank-one at most.
    \item[(ii)] $B\in\CC E_{12}\cup \CC E_{21}\setminus\{0\}$.
    \end{itemize}
\end{lemma}
\begin{proof}
(i) $\Rightarrow$ (ii). By \eqref{eq:rk-1-perp} $E_{11}^\bot=\left(\begin{smallmatrix}
    0 & \CC\\
    \CC &\CC
\end{smallmatrix}\right)$. Then, rank-one $B\in E_{11}^\bot$ if and only if $B=e_2x^\ast$ or $B=xe_2^\ast$ for a nonzero $x\in\CC$. By applying, if needed, a conjugate-linear isometry (hence, a BJ isomorphism) $X\mapsto X^\ast$, we can reduce to the later case, when $B=xe_2^\ast$. Let $x'\neq0$ be perpendicular to $x$  and consider a matrix $A=e_2e_1^\ast+x'e_2^\ast\in E_{11}^\bot$. Since $Be_2=x\perp x'=Ae_2$ we have $A\in  E_{11}^\bot\cap B^\bot$.  However, $\rank A=2$ except when $x'\in\CC e_2$, or equivalently, except when $x\in\CC e_1$ and $B\in\CC e_1e_2^\ast$, so (ii) holds.

(ii) $\Rightarrow$ (i). If $0\neq B\in\CC E_{12}$, then, by \eqref{eq:rk-1-perp}, $E_{11}^\bot\cap B^\bot=\left(\begin{smallmatrix}
    0 & 0\\
    \CC &\CC
\end{smallmatrix}\right)$ so (i) clearly holds. Likewise if $0\neq B\in\CC E_{21}$. 
\end{proof}
We apply the lemma on $B:=\Phi(E_{12})$, while keeping in mind \eqref{eq:E11fixed} and the fact that $\Phi$ strongly preserves matrices of rank-one. Consequently, $\Phi(E_{12})$ is either a scalar multiple of $E_{12}$ or a scalar multiple of $E_{21}$.
By further modifying, if needed, $\Phi$ with the isometry $X\mapsto X^\ast$  we can hence achieve that
$$\Phi(E_{11})\dot{=}E_{11}\hbox{ and  } \Phi(E_{12})\dot{=}E_{12}.$$

Notice that, by the same arguments, $E_{21}$ is also mapped to a scalar multiple of $E_{12}$ or of $E_{21}$. The former case is impossible because $E_{12}^\bot\neq E_{21}^\bot$, so $E_{12}^\bot=\Phi(E_{12})^\bot\neq\Phi(E_{21})^\bot$. Hence, also $E_{21}$ is fixed (modulo scalars). Then,   $\CC E_{22}=E_{11}^\bot\cap E_{12}^\bot\cap E_{21}^\bot$, so also $E_{22}$ is fixed  (modulo scalars). 

It follows that $\Phi$ also fixes the set of matrices living in the first/second column/row (these sets coincide with $E_{ij}^\bot\cap E_{ii}^\bot$). It also fixes the set of diagonal (=  $E_{12}^\bot\cap E_{21}^\bot$) and the set of anti-diagonal (= $E_{11}^\bot\cap E_{22}^\bot$) matrices. Moreover, $\Phi$ also preserves unitary matrices, modulo the scalars, because these coincide with right-symmetric points (see~\cite[Theorem 2.5]{turnsek2005} or \cite[Lemma 3.7]{simple}).  In particular, $\Phi(I), \Phi(J)$ are (scalar multiples of) diagonal and anti-diagonal  unitary matrices, respectively; here $$J:=\left(\begin{matrix}
    0 & 1\\
    1&0
\end{matrix}\right).$$ By modifying $\Phi$ with $X\mapsto U\Phi(X)V^\ast$ for diagonal unitary $U,V$ we can achieve that $\Phi(I)$ and $\Phi(J)$ are fixed, modulo the scalars --- in fact, if  $\Phi(I)=\left(
\begin{smallmatrix}
 a & 0 \\
 0 & d \\
\end{smallmatrix}
\right)$ and $\Phi(J)=\left(
\begin{smallmatrix}
 0 & b \\
 c & 0 \\
\end{smallmatrix}
\right)$ then we achieve this with unitary $U=\left(
\begin{smallmatrix}
 1 & 0 \\
 0 & \frac{\sqrt{a} \sqrt{b}}{\sqrt{c}
   \sqrt{d}} \\
\end{smallmatrix}
\right),V=\left(
\begin{smallmatrix}
 1 & 0 \\
 0 & \frac{\sqrt{a} \sqrt{c}}{\sqrt{b}
   \sqrt{d}} \\
\end{smallmatrix}
\right)$. Notice that this modification still preserves the matrix units $E_{ij}$, modulo the scalars.

In particular, the columns $\CC^2\otimes  e_1$ and $\CC^2 \otimes e_2$ and the rows $e_1\otimes\CC^2$ and $e_2\otimes \CC^2$ remains to be  preserved.

Let us record these observations for future use:
\begin{lemma}\label{lem:summary} A BJ isomorphism $\Phi\colon M_n(\CC)\to M_n(\CC)$ strongly preserves rank-one. Moreover, if $\Phi(E_{11})\dot{=}E_{11}$, then there exists an isometry $\dagger\colon M_2(\CC)\to M_2(\CC)$ which is either the identity or the adjoint, and diagonal unitary matrices $U,V$ such that a BJ isomorphism  $U\Phi(\cdot)^\dagger V$ fixes  $E_{ij}$, $I$, and $J$, modulo a scalar multiple. \qed
\end{lemma}

\begin{lemma}\label{lemorthogonality-of-xz-and-yz}
    Matrices $x z^\ast$ and $y z^\ast$ are BJ orthogonal if and only if $x^\ast y=0$.
\end{lemma}
\begin{proof}
  Follows directly from~\ref{eq:rk-1-perp}.
\end{proof}
It follows that there exist two bijections $\phi_i\colon\CC^2\to\CC^2$, which in both directions map perpendicular pairs of vectors to perpendicular pairs (and consequently also satisfy $\phi_i(\CC x)=\CC \phi_i(x)$), such that 
$$\Phi(xe_i^\ast)=\phi_i(x)e_i^\ast; \qquad (i=1,2).$$
We next show that $\phi_1\dot{=}\phi_2$ (i.e., $\phi_1=\phi_2$ modulo a scalar valued function).
\begin{lemma}
The following are equivalent for rank-one matrices $A=\left(\begin{smallmatrix}
    \alpha &0\\
    \beta &0
\end{smallmatrix}\right)$ and $B=\left(\begin{smallmatrix}
   0& \gamma \\
   0& \delta 
\end{smallmatrix}\right)$:
\begin{itemize}
    \item[(i)] $A^\bot\cap B^\bot$ consists of rank-one matrices only.
    \item[(ii)] $(\alpha,\beta )\in\CC(\gamma ,\delta)\setminus\{0\}$.
\end{itemize}
\end{lemma}
\begin{proof}
 Write $A=(\alpha e_1+\beta e_2)e_1^\ast$ and  $B=(\gamma e_1+\delta e_2)e_2^\ast$; by \eqref{eq:rk-1-perp} $$A^\bot\cap B^\bot=\left\{\begin{pmatrix}
      -\lambda\bar{\beta} &-\mu\bar{\delta}\\
     \lambda \bar{\alpha} & \mu\bar{\gamma}
  \end{pmatrix};\; \lambda,\mu\in\CC\right\}.$$ This consists of matrices of rank-one at most if and only if the two columns are linearly dependent (i.e., are parallel), or equivalently, if and only if their orthogonal complements, $\alpha e_1+\beta e_2$ and $\gamma e_1+\delta e_2$ are parallel, as claimed by (ii).
\end{proof}
We apply this to rank-one matrices $A=\Phi(\left(\begin{smallmatrix}
    \alpha &0\\
    \beta &0
\end{smallmatrix}\right))=\phi_1(\alpha e_1+\beta e_2)e_1^\ast$ and $B=\Phi(\left(\begin{smallmatrix}
       0 &\alpha\\
       0 &\beta
\end{smallmatrix}\right))=\phi_2(\alpha e_1+\beta e_2)e_2^\ast$
to deduce that $\phi_1=\phi_2=:\phi$, modulo a scalar multiple.

Similar conclusion (e.g, by applying the previous arguments on BJ isomorphism $\Psi(X)=\Phi(X^\ast)^\ast$) can be shown for matrices which live only in the first  or only in the second row: here, $\psi_1\dot{=}\psi_2=:\psi$ is such that $\Phi(e_ix^\ast)=e_i\psi_i(x)^\ast\in\CC \,e_i\psi(x)^\ast$.

\begin{lemma}\label{lem:M_2-phi}
    There exists a scalar-valued function $\gamma$ such that the restriction of $\Phi$ to rank-one matrices takes the form
    $$\Phi( xy^\ast) =\gamma(xy^\ast)\,\phi(x)\psi(y)^\ast.$$
\end{lemma}
\begin{proof}
Let $\tilde{x},\tilde{y}\in\CC^2$ be perpendicular to $x,y$, respectively. By~\eqref{eq:rk-1-perp}, $(\tilde{x}e_i^{\ast})^\bot$ is the set of all matrices $T$ with $Te_i\in\CC x$, so  $(\tilde{x}e_1^{\ast})^\bot\cap (\tilde{x}e_2^{\ast})^\bot$ equals those $T$ with image contained in $\CC x$. Due to $\mathrm{Im}(T^\ast)=(\Ker T)^\bot$, it easily follows that
$$\CC\,xy^\ast = (\tilde{x}e_1^\ast)^\bot\cap (\tilde{x}e_2^\ast)^\bot\cap (e_1\tilde{y}^\ast)^\bot\cap (e_2\tilde{y}^\ast)^\bot .$$ Applying $\Phi$ and using a straightforward consequence of~\eqref{eq:rk-1-perp} that two rank-one $R_1,R_2$ are linearly dependent if and only if $R_1^\bot = R_2^\bot$  (so $\Phi(\CC xy^\ast)=\CC \Phi(xy^\ast)$), we see that
\begin{align*}
\CC\Phi(xy^\ast)&=\Phi(\tilde{x}e_1^\ast)^\bot\cap \Phi(\tilde{x}e_2^\ast)^\bot\cap \Phi(e_1\tilde{y}^\ast)^\bot\cap \Phi(e_2\tilde{y}^\ast)^\bot\\
&=(\phi(\tilde{x})e_1^\ast)^\bot\cap (\phi(\tilde{x})e_2^\ast)^\bot\cap (e_1\psi(\tilde{y})^\ast)^\bot\cap (e_2\psi(\tilde{y})^\ast)^\bot.
\end{align*}
 Now, since $\phi\colon\CC^2\to\CC^2$ maps perpendicular pair $(x,\tilde{x})$ to perpendicular pair, we get  
$$(\phi(\tilde{x})e_1^\ast)^\bot=\CC\phi(x)e_1^\ast +\CC^2 \otimes e_2,$$ and likewise for other sets in the intersection. Hence, the intersection of the first two sets is $\phi(x)(\CC e_1+\CC e_2)^\ast$. Combining with the intersection of the last two sets we easily derive the result. 
\end{proof}
We will  denote shortly $\Phi|_{{\mathcal R}_1}=\gamma(\cdot)\,\phi\otimes \psi$. 

\begin{lemma}\label{lem:phi=psi}
   $\phi=\psi$, modulo multiplication by a scalar valued function. Consequently,
   $$\Phi|_{{\mathcal R}_1}=\gamma(\cdot)\phi\otimes\phi.$$
\end{lemma}
\begin{proof}
Notice that $I$ is fixed, and a rank-one $xy^\ast \perp I$ if and only if $x^\ast y=0$. This is further equivalent to $\gamma(xy^\ast)\phi(x)\psi(y)^\ast=\Phi(xy^\ast)\perp\Phi(I)=I$, so that $x\perp y$  if and only if $\phi(x)\perp\psi(y)$. Now,  fix any nonzero $x$ and let $y=x'$ be perpendicular to it.  Then, $\phi(x)\perp\psi(x')$ but since $\phi$ maps perpendicular pairs to perpendicular pairs, and $\phi(x')$ is the only vector perpendicular to $\phi(x)$ (modulo scalar multiplication), we have $\psi(x')=\phi(x')$ (again modulo scalar multiplication), as claimed.
\end{proof}
Also, in the sequel we will assume, as we clearly may,  that $\phi$ maps normalized vectors into normalized  vectors.
\begin{lemma}\label{lem:M2(CC)-diag(1-1)}
$\Phi\left(\begin{smallmatrix}
    1 &0\\
    0 & -1
\end{smallmatrix}\right)=\left(\begin{smallmatrix}
    1 &0\\
    0 & -1
\end{smallmatrix}\right)$, up to a scalar multiple. Consequently, for every $\theta\in\RR$ there is $\tau \in\RR$ such that $\phi$  maps $e_1+e^{i\theta} e_2$ into $\CC(e_1+e^{i\tau} e_2)$.
\end{lemma}
\begin{proof}
Notice that $\left(\begin{smallmatrix}
    1 &0\\
    0 & -1
\end{smallmatrix}\right)$ is a diagonal unitary matrix. Such matrices are mapped onto themselves by $\Phi$ modulo a scalar multiple (scalar multiples of unitary matrices are exactly the right-symmetric elements in $M_2(\CC)$, by~\cite[Theorem 2.5]{turnsek2005}). Notice also that a diagonal unitary $D$ belongs to $\CC\left(\begin{smallmatrix}
    1 &0\\
    0 & -1
\end{smallmatrix}\right)$ if and only if the sum of its eigenvalues vanishes, or equivalently, if and only if $0\in W(D)$, the numerical range of $D$. This is further equivalent to the existence of a nonzero $x$ with $x^\ast Dx=0$ or, equivalently, to $xx^\ast\perp D$. Now apply $\Phi$ to deduce that there exists nonzero $y=\phi(x)$ such that $yy^\ast$ is orthogonal to a multiple of diagonal unitary $\Phi(D)$. Thus, $\Phi(D)\in\CC\left(\begin{smallmatrix}
    1 &0\\
    0 & -1
\end{smallmatrix}\right)$.
The last claim follows since, after a trivial calculation, each nonzero vector $x$ with $x^\ast \left(\begin{smallmatrix}
    1 &0\\
    0 & -1
\end{smallmatrix}\right)x=0$ takes the form $x\in\CC(e_1+e^{i\theta} e_2)$ for some $\theta\in\RR$.
\end{proof}

Now, the vector of the form $x=e_1+e^{i\theta} e_2$ satisfies $xx^\ast\bot J$ if and only if $e^{i\theta}=\pm i$. Since $\Phi$ fixes $J$, we get that the two vectors $e_1\pm i e_2$ are mapped  to each other (modulo scalars).

Within the next  lemma we will be using a generalized version of SVD, $A=\sigma_1 xy^\ast+\sigma_2 x'(y')^\ast$ whereby $x,x'$ and $y,y'$ form, as usual, an orthonormal basis, however,  we allow $\sigma_1,\sigma_2$ to be any complex numbers instead of usually $\sigma_1\ge \sigma_2\ge 0$. Such decomposition will be briefly denoted as gSVD.
\begin{lemma}\label{lem:M_2(C)-A} Let $\phi\colon\CC^2\to\CC^2$ be a bijection which maps orthonormal pairs onto orthonormal pairs and such that $\Phi(xy^\ast)\dot{=}\phi(x)\phi(y)^\ast$.
    Let $A=xy^\ast+\sigma x'(y')^\ast$ be a gSVD, with $|\sigma|<1$. Then, there exists $|\tau|<1$ such that 
    $$\Phi(A)\dot{=}\phi(x)\phi(y)+\tau \phi(x')\phi(y')^\ast.$$
\end{lemma}
\begin{proof}
Since $|\sigma|<1$, then $A$ is not a multiple of unitary, hence not right symmetric, so $\Phi(A)$ is also not a multiple of unitary. Recall that if $B=uv^\ast+\sigma u' (v')^\ast$ is any matrix in  its gSVD which is  not a multiple of unitary, then $B^\bot=(uv^\ast)^\bot$.
Thus, 
$$A^\bot=(xy^\ast)^\bot,$$
and applying $\Phi$ gives
$$\Phi(A)^\bot = (\phi(x)\phi(y)^\ast)^\bot.$$
Since two rank-one matrices share the same outgoing neighborhood if and only if they are scalar multiples of each other, we see that the gSVD of $\Phi(A)$ equals
$$\Phi(A)\dot{=}\phi(x)\phi(y)^\ast +\tau \phi(x')\phi(y')^\ast $$
for some $|\tau|<1$, where we used the fact that $\phi$ maps orthonormal pairs $(x,x')$ and $(y,y')$  into orthonormal pairs. 
\end{proof}

Recall that $\phi(\CC x)=\CC x$, so $\phi$ induces a well-defined bijection $\tilde{\phi}$ on  a projective space $\PP(\CC^2):=\{[x]:=\CC x;\;\;x\in\CC^2\setminus\{0\}\}$.  We next show that the induced map strongly preserves the (quantum) angle $\frac{\pi}{4}$ among the lines in $\CC^2$. This will enable us to use Theorem 2.3. from \cite{GeherIMRN2020} for $n=2$, where such maps were classified. We remark that that there is a recent  upgrade of this theorem in case when $n\geqslant3$ and the angle is between $\left(\frac{\pi}{4},\frac{\pi}{2}\right)$; it  can be found in \cite{GeherMoriIMRN2022}.

\begin{lemma}
There exists a (conjugate)linear unitary $U$,  a scalar-valued function $\gamma\colon\CC^2\setminus\{0\}\to\CC\setminus\{0\}$, and a subset $\Omega\subseteq\CC^2\setminus\{0\}$ such that $\phi(x)=\gamma_x Ux$ for $x\in\Omega$  and $\phi(x)=\gamma_x Ux'$ for $x\notin\Omega$; here $x'$ is a nonzero vector perpendicular to $x$. 

\end{lemma}
\begin{proof}
 We claim that the lines $[x]:=\CC x,[y]:=\CC y$ meet at an angle $\frac{\pi}{4}$ (that is, $\frac{|x^\ast y| }{\|x\|\cdot\|y\|}=\frac{1}{\sqrt{2}}$) if and only if the same holds for their images $\phi([x]),\phi([y])$.

To this end we can scale  their representatives $x,y$ to achieve that they are normalized and that $x^\ast y=\frac{1}{\sqrt{2}}$. Then,  there exists a unitary $U$ such that $Ux=e_1$ and  $Uy=\frac{e_1+e_2}{\sqrt{2}}$ (the first column of $U^\ast $ is $x$ and the second one is $\sqrt{2}y-x$; one computes that this is a unitary with the desired properties).
Notice that the rank-one matrix $U^\ast E_{11} U$ is mapped by BJ isomorphism $\Phi$ again into rank-one, so there exist unitary $U_1,V_1$ such that $U_1\Phi(U^\ast E_{11} U)V_1^\ast\dot{=} E_{11}$. Consider the BJ isomorphism  
$$\Phi_1\colon X\mapsto U_2\Bigl(\bigl(U_1\Phi(U^\ast XU)V_1^\ast\bigr)^\dagger\Bigr)V_2^\ast$$
where $\dagger$ denotes either the identity map or the adjoint map, and will be specified in a moment, and where $U_2,V_2$ are diagonal unitary matrices, which will also be specified in a moment. Notice that, whatever choice we make for the map $\dagger$ and for diagonal unitaries $U_2,V_2$, $\Phi_1$ will always  map $E_{11}$ into $\CC E_{11}$.  So, by Lemma~\ref{lem:n=2-E_(11)}, it will also map $E_{12}$ either into $\CC E_{12}$ or into $\CC E_{21}$. With appropriate choice of $\dagger$ we achieve that $\Phi_1(E_{12})\in\CC E_{12}$. Then, as shown in a text between the lemmata \ref{lem:n=2-E_(11)} and~\ref{lemorthogonality-of-xz-and-yz} (c.f. also a summary in Lemma~\ref{lem:summary}), $\Phi_1$ fixes every $E_{ij}$, modulo a scalar multiple, and hence maps diagonal/antidiagonal matrices onto itself and we can prescribe diagonal unitaries  $U_2,V_2$ such that $\Phi_1$ fixes matrices $I$ and $J$, modulo scalar multiples.

By lemmata~\ref{lem:M_2-phi}--\ref{lem:M2(CC)-diag(1-1)} there exists a bijection $\phi_1\colon\CC^2\to\CC^2$, which takes normalized vectors to normalized vectors, maps $e_1,e_2$ into their scalar multiples (due to $\Phi_1(E_{ij})\dot{=}E_{ij}$) and maps $e_1+e^{i\theta}e_2$ into a scalar multiple of $e_1+e^{i\tau}e_2$, such that
$$\Phi_1(uv^\ast)\dot{=}\phi_1(u)\phi_1(v)^\ast.$$
It follows that
\begin{align*}
e_1\tfrac{(e_1+e^{i\tau} e_2)^\ast}{\sqrt{2}}
&\dot{=}\phi_1(e_1)\phi_1(\tfrac{e_1+e_2}{\sqrt{2}})^\ast\dot{=}\Phi_1(e_1\tfrac{(e_1+e_2)^\ast}{\sqrt{2}})\\
&=U_2(U_1\Phi(U^\ast e_1\tfrac{(e_1+e_2)^\ast}{\sqrt{2}} U)V_1^\ast)^\dagger V_2^\ast=U_2(U_1\Phi(xy^\ast)V_1^\ast)^\dagger V_2^\ast\\
&\dot{=}
U_2(U_1\phi(x)\phi(y)^\ast V_1^\ast)^\dagger V_2^\ast.
\end{align*}
Since $uv^\ast =zw^\ast$ if and only if $u,z$ and $v,w$ are linearly dependent, and since $U_i,V_j$ are unitaries, we see that
 $|\phi(x)^\ast \phi(y)|=\left|e_1^\ast\tfrac{(e_1+e^{i\tau} e_2)}{\sqrt{2}}\right|=\frac{1}{\sqrt{2}}$ (if $\dagger$ is identity map), or else $|\phi(y)^\ast \phi(x)|=\left|e_1^\ast\tfrac{(e_1+e^{i\tau} e_2)}{\sqrt{2}}\right|=\frac{1}{\sqrt{2}}$ (if $\dagger$ is adjoint map).
 In both cases, the initial assumption $|x^\ast y|=\frac{1}{\sqrt{2}}$ implies $|\phi(x)^\ast\phi(y)|=|x^\ast y|=\frac{1}{\sqrt{2}}$.

 The converse implication, that from $|\phi(x)^\ast\phi(y)|=\frac{1}{\sqrt{2}}$ we can derive $|x^\ast y|=\frac{1}{\sqrt{2}}$, follows by considering the BJ isomorphism $\Phi^{-1}$.

 Now $\phi$ induces a well-defined projective-space bijection
 $\tilde{\phi}\colon\PP(\CC^2)\to\PP(\CC^2)$, which obviously strongly preserves the angle $\frac{\pi}{4}$. Such bijections were classified by Geh\'er in  \cite[Theorem 2.3]{GeherIMRN2020} (in a generalization of a paper by  Li-Plevnik-\v Semrl~\cite{li2012preservers}), where they considered the same problem on Hilbert spaces of dimension at least five.  The result hence follows by Geh\'er's \cite[Theorem 2.3]{GeherIMRN2020}.
\end{proof}

We can replace $\Phi$ by a map $X\mapsto U^\ast\Phi(X)U$ to achieve that  $\phi(x)\dot{=} x$ for $x\in\Omega$ and  $\phi(x)\dot{=}x'$ for $x\notin\Omega$. 
Notice also that orthocomplementation is a conjugate-linear isometric  map on $\CC^2$, given by
$$\Pi(\alpha e_1+\beta e_2)=-\overline{\beta} e_1+\overline{\alpha}e_2.$$
It is easily seen that its adjoint equals   $\Pi^\ast=-\Pi$ so is also a conjugate-linear isometry. As such, $\Pi$  induces a conjugate-linear  BJ isomorphism $X\mapsto\Pi X\Pi^\ast$ on $M_2(\CC)$ and by composing it with $\Phi$ we get a BJ isomorphism which maps $xy^\ast$ into $(\Pi\circ\phi)(x)((\Pi\circ\phi)(y))^\ast$ (modulo a scalar multiple). This way (that is, by replacing, if needed, $\Phi$ with $\Pi\circ\Phi\circ \Pi^\ast$)  we can achieve 
that
$$\phi(e_i)\dot{=}e_i,\;(i=1,2),\quad\hbox{ and }\quad\phi(x)\dot{=}\begin{cases}x; & x\in\Omega\\
x'; & x\notin\Omega\end{cases}.$$
Observe that $\Phi(e_ie_j^\ast)\dot{=}\phi(e_i)\phi(e_j)^\ast\dot{=}e_ie_j$.

\begin{lemma}\label{lem:phi-identity} $\phi$ is the  identity  (modulo a scalar-valued function).
\end{lemma}
\begin{proof}
{\bf Case~1}. There exist a unimodular number $\mu=e^{i\xi},\xi\in\mathbb{R}$ such that $x:=e_1+\mu e_2$ satisfies $$\phi(x)\dot{=}x.$$
Consider any 
\begin{equation}\label{eq:y}
    y=e_1+se^{i\theta} e_2;\qquad 
s>0, s\neq 1,\;\hbox{ and }\;\theta \in\mathbb{R}.
\end{equation}
We claim that $\phi(y)\dot{=}y$. 

To verify the claim  assume otherwise that  $\phi(y)\dot{=}y'$, the orthocomplement of $y$. Suppose first that $s>1$, define $\sigma:=\frac{-e^{i(\xi-\theta)}}{s}$ (so  $|\sigma|<1$) and  form $$A=\left(\begin{smallmatrix}
    1 &0\\
    0 & \sigma
\end{smallmatrix}\right)= e_1e_1^\ast+ \sigma e_2e_2^\ast.$$
Notice that $x^\ast Ay=0$, or equivalently, that $xy^\ast \perp A$. Applying BJ isomorphism~$\Phi$ this gives $x(y')^\ast\perp \Phi(A)$, or equivalently, $x^\ast\Phi(A)y'=0$. Now, by Lemma~\ref{lem:M_2(C)-A}, $\Phi(A)\dot{=}\phi(e_1)\phi(e_1)^\ast + \tau \phi(e_2)\phi(e_2)^\ast=\left(\begin{smallmatrix}
    1 &0\\0&\tau
\end{smallmatrix}\right)$   for some $\tau$ with $|\tau|< 1$. Since $y'=-s e^{-i\theta} e_1+e_2$, one easily computes  that  
$$0=x^\ast \Phi(A)y'=-s e^{-i\theta}+\tau  e^{-i\xi} $$
so that $|\tau|=|s|>1$, a contradiction. 

Suppose next $0<s<1$. Then, by considering $\left(\begin{smallmatrix}
     \frac{1}{\sigma} &0\\
    0 &1
\end{smallmatrix}\right)$ instead of the previous $A$ (with $\sigma:=\frac{-e^{i(\xi-\theta)}}{s}$ same as before), one again gets a contradiction.

{\bf Case~2}. There exist a unimodular number $\mu$ such that $x:=e_1+\mu e_2$ satisfies $$\phi(x)\dot{=}x'=e_1-\overline{\mu}e_2.$$
Following the previous arguments  one sees that it is impossible to have  $\phi(y)\dot{=}y$ for some $y$ of the form~\eqref{eq:y}. So that, in this case, $\phi(y)\dot{=}y'$ for every such $y$.

\medskip

The two cases also show that if for some unimodular $\mu$ we have $\phi(x_{\mu})\dot{=}x_{\mu}$, where $x_{\mu}=e_1+\mu e_2$, then this holds for every unimodular $\mu$: In fact, 
by Case~1, $\phi(e_1+2 e_2)\dot{=}e_1+2e_2$. But, if, for some other unimodular $\nu$ we would have $\phi(x_{\nu})\dot{=}x_{\nu}'$, then, by Case~2, $\phi(e_1+2 e_2)\dot{=}(e_1+2e_2)'=e_1-\frac{1}{2} e_2$. The two are incompatible.

 Therefore, either $\phi(e_1+r e^{i\xi} e_2) \dot{=}e_1+r e^{i\xi} e_2$ for every $r>0$ and every $\xi\in\mathbb{R}$  or else $\phi(e_1+r e^{i\xi} e_2) \dot{=}(e_1+r e^{i\xi} e_2)'$ for every $r>0$ and every $\xi\in\mathbb{R}$.  In the former case, by our initial assumption of $\phi(e_i)\dot{=}e_i$, we see that $\phi(z)\dot{=}z$ for every $z\in\CC^2$, so $\phi$ is an identity (modulo a scalar-valued function).

In  the later case, we temporarily once again compose BJ isomorphism $\Phi$ by orthocomplementation $\Pi\circ \Phi(\cdot)\circ\Pi^\ast$ to achieve that $\phi(z)\dot{=}z$  for every $z\in\CC^2\setminus(\CC e_1\cup\CC e_2)$, while $\phi(e_1)\dot{=}e_1'=e_2$ and $\phi(e_2)\dot{=}e_2'=-e_1$. We claim this is impossible:
  Consider here $A=uv^*$, where $u=\left(\frac{3}{5},\frac{4}{5}\right)^\ast$ and $v=\left(\frac{5}{13},\frac{12}{13}\right)^\ast$. Then $uv^*\perp (e_1e_1^*+\lambda e_2e_2^*)$ if and only if $\lambda=-\frac{5}{16}$. Considering their $\Phi$-images we would have, by Lemma~\ref{lem:M_2(C)-A}, $uv^*\perp e_2e_2^*+\sigma e_1e_1^*$ for some $|\sigma|<1$. But the only $\sigma$ that satisfies this relation is $\sigma=-\frac{16}{5}$ which is in contradiction with $|\sigma|<1$.
\end{proof}

 \begin{proof}[Proof of Theorem~\ref{thm:simplecase} for $n=2$]
   By Lemma~\ref{lem:0classifion} and lemmata~\ref{lem:summary}--\ref{lem:phi-identity} we can compose $\Phi$ with (conjugate)linear isometries $X\mapsto X^\ast$ and  $X\mapsto UXV^{\ast}$ where $U,V$ are either both unitary or are both conjugate-linear unitary, to achieve that the changed map, which we again denote by $\Phi$, is a BJ isomorphism that fixes every  matrix of rank at most one, modulo  a scalar multiple. The result then follows from Lemma~\ref{lem:A=B}, applied on $B=\Phi(A)$.
  \end{proof}
  \bigskip

\section{Proof of the main theorem}
 
 To prove the first theorem, we will require an additional lemma.

 \begin{lemma}\label{lem:gamma1=gamma2}    Let $\A=\bigoplus_1^\ell M_{n_i}(\CC)$ and let $x_ix_i^\ast\in 0\oplus M_{n_i}(\CC)\oplus 0$ and $x_jx_j^\ast\in 0\oplus M_{n_j}(\CC)\oplus 0$ be norm-one rank-one matrices. If $i\neq j$,  then there exists a unique unimodular number $\mu=-1$ such that $(x_ix_i^\ast\oplus\mu x_jx_j^\ast)\perp I$. 
 \end{lemma}
 \begin{proof}
    We can consider $\A$ embedded into $M_N(\CC)$, $N=n_1+\dots+n_\ell$. Consider $x_i$ as vector in $0\oplus \CC^{n_i}\oplus 0$ and similarly for $x_j$. Then, from $i\neq j$ we know that $x_i$ and $x_j$ are perpendicular.  Since $(x_ix_i^\ast\oplus \mu x_jx_j^\ast)=: R\perp I$, there exist a normalized vector $z$, where $R$ achieves its norm, so that $\langle Rz,Iz\rangle=z^\ast Rz=0$. Now, $R$ achieves its norm on $\Span\{x_i,x_j\}$ and since $z$ is normalized and $x_i,x_j$ are orthonormal, there exist $c,s\in\CC$ such that 
 $$z=c x_i\oplus s x_j;\qquad |c|^2+|s|^2=1.$$ 
 Then, 
 $$z^\ast Rz=(cx_i\oplus sx_j)^\ast (cx_i \oplus \mu sx_j) =|c|^2(x_i^\ast x_i)+\mu|s|^2(x_j^\ast x_j),$$
so $z^\ast Rz=0$  if and only if $\mu=-\frac{|c|^2}{|s|^2}< 0.$
Since $\mu$ is unimodular, $\mu=-1$.
\end{proof}

\begin{proof}[Proof of Theorem~\ref{thm:genera}.]
Since a finite-dimensional $C^*$-algebra $\A$ has no abelian summand we have $2\le \dim\A<\infty$.
Within our recent paper~\cite[Theorem 1.1]{kuzma_singla_2025} we showed that in this case, if $\Phi\colon\A\to\B$ is a BJ isomorphism between complex $C^*$-algebras, then $\A$ and $\B$ are isometrically $\ast$-isomorphic and we can identify  $\A=\B$. Moreover, by using the properties of smooth and of relative left-symmetric points, we further showed in~\cite[Corollary 2.9]{kuzma_singla_2025}  that $\Phi$ maps elements which belong to a single block of $\A$ onto matrices which belong to a single block of $\B$, and the sizes of the blocks are the same. Consequently,  by writing 
$$\A=\bigoplus_1^\ell M_{n_i}(\CC),$$
and identifying individual blocks with $0\oplus M_{n_i}(\CC)\oplus 0=M_{n_i}(\CC)$, there exists a permutation  of block constituents $\sigma\colon\{1,\dots,\ell\}\to\{1,\dots,\ell\}$, such that
$$n_i=n_{\sigma(i)}\quad\hbox{ and }\quad
\Phi_i:=\Phi|_{M_{n_i}(\CC)}\colon M_{n_i}(\CC)\to M_{n_{\sigma(i)}}(\CC).$$
By~\cite[Corollary 2.9]{kuzma_singla_2025}, $\Phi_i$ is a bijection; by the  definition of BJ orthogonality it   clearly strongly preserves  BJ orthogonality. Such BJ isomorphisms were classified within  Theorem~\ref{thm:simplecase} and take the form~\eqref{eq:conjugate-linear_isometries}, modulo multiplication by a scalar-valued function  (note  $n_i\ge 2$ since $\A$ has no abelian summand).  Observe also that 
the permutation of indices $\sigma$ induces an isometry on $\A=\B$  by
$$\sigma\colon\bigoplus X_i\mapsto \bigoplus X_{\sigma(i)}.$$
Hence, by Theorem~\ref{thm:simplecase}, there exist  isometries $\Psi_i\colon M_{n_i}(\CC)\to M_{n_i}(\CC)$ such that, upon identifying $M_{n_i}(\CC)=M_{n_{\sigma(i)}}(\CC)$, we have
$\Phi_i(X)\dot{=}\Psi_i(X)$. 
Then, 
$$\Psi:=\sigma\circ(\Psi_1\oplus\dots\oplus\Psi_\ell)^{-1}$$ is an isometry of $\A$ and $\Psi^{-1}\circ\Phi$  is a BJ isomorphism which fixes every rank-one element of $\A$, modulo a scalar multiple. We denote it again by $\Phi$. 

Choose now an arbitrary $A=\bigoplus A_i\in\A$ and let $$B=\bigoplus B_i:=\Phi(A).$$ Applying Lemma~\ref{lem:A=B} on each individual block reveals that $B_i=\gamma_i A_i$ for some nonzero scalar $\gamma_i=\gamma_i(A_i)$. 
Therefore
$$\Phi(A)=\Gamma(A)A$$
for some function $\Gamma\colon \A\to Z(\A)=\bigoplus\CC I_{n_i}$, the center of $\A$, such that $\Gamma(A)$ is invertible for every $A$. It remains to show that $\Gamma(A)=\gamma(A)P(A)$, where $\gamma(A)\in\CC$ is unimodular, and $P(A)\in Z(\A)$ is positive definite. To ease the flow of arguments, we proceed in steps.

{\bf Step 1}. \textit{$A$ and $B=\Phi(A)$ attain their norm on the same summands. Moreover, if those summands are indexed by $i_1,\dots, i_k$, then $|\gamma_{i_1}(A)|=\dots=|\gamma_{i_k}(A)|$.} 

To see this, we assume for simplicity $(i_1,\dots,i_k)=(1,\dots,k)$, i.e.,  $\|A_{1}\|=\dots=\|A_{k}\|=\|A\|$ and $\|A_i\|<\|A\|$ for $i\geq k+1$. Then, $A\not\perp X:=(A_1\oplus\dots\oplus A_k)\oplus 0$, so also
$$\bigoplus\gamma_i(A)A_i=\Phi(A)\not\perp\Phi(X)=\gamma_1(X) A_1\oplus\dots\oplus\gamma_k(X)A_k\oplus 0.$$
It easily follows that $\Phi(A)$ cannot attain its norm on some block with index greater than $k$, that is, 
$$\|\gamma_i(A) A_i\|<\|\Phi(A)\|\quad\hbox{ for }i\ge k+1.$$
Conversely, if, say, $\|\gamma_k(A) A_k\|<\|\Phi(A)\|$, then  $$\Phi(A)\not\perp (\gamma_1A_1\oplus\dots\oplus\gamma_{k-1}A_{k-1}\oplus0)=:{Y}.$$ By surjectivity there exists $X\in\A$ with $\Phi(X)=Y$, and clearly 
$$X=\alpha_1 A_1\oplus\dots\oplus \alpha_{k-1} A_{k-1}\oplus 0$$ 
for some nonzero scalars $\alpha_i$.  However, since $A$ attains its norm also on $k$-th summand,  $A\perp X$, a contradiction. 

{\bf Step 2}. With $A=I$ this shows that $\Phi(I)={\bm \gamma}(I) I$ where ${\bm \gamma}(I)=\bigoplus\gamma_i(I)I_{n_i}$ is a scalar multiple of a unitary. By replacing $\Phi$ with a BJ isomorphism ${\bm \gamma}(I)^\ast\Phi(\cdot)$ we can hence assume that 
$$\Phi(I)=I.$$

{\bf Step 3}. \textit{If $R=0\oplus z_iz_i^\ast\oplus -z_jz_j^\ast\oplus0$ where $\|z_i\|=\|z_j\|=1$, then $\Phi(R)=\gamma(R)R$, for some $\gamma(R)\in\CC$.}

Choose normalized $z_iz_i^\ast\in 0\oplus M_{n_i}(\CC)\oplus 0$ and $z_jz_j^\ast\in 0\oplus M_{n_j}(\CC)\oplus 0$ with $i\neq j$ and for $R=0\oplus (z_iz_i^\ast)\oplus(-z_jz_j^\ast)\oplus 0$. Since $R$ attains its norm on $z=z_i\oplus z_j$ (we omitted zero summands), and 
$$\langle Rz,I z\rangle=\bigl\langle z_i\oplus (-z_j)\,,\,z_i\oplus z_j\bigr\rangle=0$$
so $R \perp I$. Applying $\Phi$  we deduce by Lemma~\ref{lem:gamma1=gamma2}, that $\gamma_i(R)=\gamma_j(R)=:\gamma(R)\in\CC$, as claimed.

{\bf Step 4}. \textit{If $\|A_i\|=\|A_j\|=\|A\|$ and their numerical ranges satisfy $W(A_i)\cap W(A_j)\setminus\{0\}\neq\emptyset$, then $\gamma_i(A)=\gamma_j(A)$.}

By definition, there exist normalized vectors $z_i$ and $z_j$ (corresponding to blocks containing $A_i$ and $A_j$) such that $\lambda:=z_i^\ast A_iz_i=z_j^\ast A_j z_j\in W(A_i)\cap W(A_j)\setminus\{0\}$. Choose $R$ and $z$ as in Step 3. Then from
$$\overline{\langle Rz,Az\rangle}=\overline{\bigl\langle z_i\oplus (-z_j)\,,\,A_iz_i\oplus A_jz_j\bigr\rangle} =z_i^\ast A_i z_i-z_j^\ast A_j z_j=0$$
we have $R\perp A$. Applying $\Phi$, which fixes $R$ modulo scalars by Step 3, we get
$R\dot{=}\Phi(R)\perp\Phi(A)=\bigoplus \gamma_i(A) A_i$. Now, $R$ achieves its norm only on $\Span\{z_i,z_j\}$  and since $z_i,z_j$ are orthonormal vectors (they correspond to different blocks), there exist $c,s\in\CC$, $|c|^2+|s|^2=1$, such that
\begin{align*}
    0&=\langle R(cz_i\oplus sz_j),\Phi(A)(cz_i\oplus sz_j)\rangle=\langle cz_i\oplus (-sz_j), \gamma_i(A)A_icz_i\oplus \gamma_j(A)A_jsz_j\rangle\\
    &=|c|^2 z_i^\ast (\gamma_i(A)A_i)^\ast z_i-|s|^2z_j^\ast (\gamma_j(A)A_j)^\ast z_j\\
    &=|c|^2\overline{\lambda\gamma_i(A)} -|s|^2\overline{\lambda\gamma_j(A)}.
\end{align*}
After simplifying (while keeping in mind that $\lambda \gamma_k(A)\neq0 $) we get
$\frac{\gamma_i(A)}{\gamma_j(A)}=\frac{|s|^2}{|c|^2}>0$. 
Since $\|A_i\|=\|A_j\|=\|A\|$, then, by Step 1, $\gamma_i(A)=\gamma_j(A)$.

{\bf Step 5}. \textit{If $A=0\oplus x_iy_i^\ast \oplus x_jy_j^\ast\oplus 0$, where $\|x_iy_i^\ast\|=\|x_jy_j^\ast\|$, then $\Phi(A)=\gamma(A) A$ for some $\gamma(A)\in\CC$.}

Notice that $W(x_iy_i^\ast)$ is a (possibly degenerate) elliptic disc with foci at $0$ and $y_i^\ast x_i$ (the eigenvalues of $x_iy_i^\ast$) and minor axis equals to $\sqrt{\text{tr}((x_iy_i^\ast)^* x_iy_i^\ast) - |0|^2 - |y_i^\ast x_i|^2}$ (see \cite{Li1996}) and two such elliptic discs always intersect in a nonzero point, except when they are both degenerate, i.e., when $y_i\dot{=}x_i$ and $y_j\dot{=}x_j$. In the former case, we are done by Step 4. In the latter case, after omitting zero summands, $$A\dot{=}x_ix_i^*\oplus \mu x_jx_j^*$$ where $\mu$ is unimodular and $\|x_i\|=\|x_j\|$. Here we choose $R=x_ix_i^*\oplus \mu x_jz_j^*$, where $z_j\not\in\CC x_j$, $\|z_j\|=\|x_j\|=\|x_i\|$ and $$z_j^\ast x_j<0$$  (it exists since $j$-th summand of $\A$ has size at least $2$). Notice $W(\mu x_jz_j^*)$ is a nondegenerated eliptic disc, so by the first part of Step~5, $\Phi(R)\dot{=} R$. Also, $R$    attains its norm on $w=cx_i\oplus sz_j$, and we can find $c,s$ such that $|c|^2+|s|^2=1$ and $\frac{|c|^2}{|s|^2}=-z_j^*x_j>0$. One can show that $\langle Rw,Aw\rangle=0$ and hence $R\perp A$. Then also $$R\dot{=}\Phi(R)\perp\Phi(A)\dot{=}\gamma_i(A)x_ix_i^*+\gamma_j(A)\mu x_jx_j^*$$ and $|\gamma_i(A)|=|\gamma_j(A)|$ by Step 1.

Also 
\begin{equation}
    \begin{split}\nonumber
    0&=\langle R(c'x_i\oplus s'z_j)\,,\,\Phi(A)(c'x_i\oplus s'z_j)\rangle\\
    &=\langle c'x_i\oplus \mu s'x_j\,,\,c'\gamma_i(A)x_i\oplus \mu s'\gamma_j(A)x_jx_j^*z_j\rangle\\
    &=|c'|^2\overline{\gamma_i(A)}(x_i^*x_i)+|s'|^2\overline{\gamma_j(A)}(z_j^*x_j)(x_j^*x_j),
    \end{split}
\end{equation}
so $\overline{\left(\frac{\gamma_i(A)}{\gamma_j(A)}\right)}=-\frac{|s'|^2(z_j^*x_j)}{|c'|^2}>0.$ This shows that $\gamma_i(A)=\gamma_j(A)=:\gamma(A)\in\CC$.

{\bf Step 6}. Consider now the general case  when summands $A_i,A_j$ are both nonzero.
Let 
$$A_i=\sigma_1(A_i)x_iy_i^\ast+\dots $$
$$A_j=\sigma_1(A_j)x_jy_j^\ast+\dots $$
be their SVD.

Notice that 
$$x_i^\ast A_i y_i= 
\sigma_1(A_i)\neq0\quad\hbox{ and }\quad x_j^\ast A_j y_j=\sigma_1(A_j).$$ 
So, $R:=x_iy_i^\ast \oplus (-x_jy_j^\ast)$ achieves its norm on the span of $y_i,y_j$. These two vectors belong  to different summands so  they are orthonormal. In particular,~$R$ achieves its norm on a normalized vector
$$z:=cy_1\oplus sy_2,$$ 
where $c,s>0$,  $|c|^2+|s|^2=1$ are chosen so that $\frac{|s|^2}{|c|^2}=\frac{\sigma_1(A_i)}{\sigma_1(A_j)}$, and maps it into $$Rz=cx_i\oplus (-sx_j),$$
while $Az=cA_iy_i\oplus sA_jy_j$ (we are not writing the zero direct summands). So,
$$\overline{\langle Rz,Az\rangle}=|c|^2x_i^\ast A_iy_i-|s|^2x_j^\ast A_j y_j=|c|^2 \sigma_1(A_i)-|s|^2\sigma_1(A_j)=0$$
and hence $R\perp A$.

Apply $\Phi$ to get, by Step 5,
$$R\dot{=}\Phi(R)\perp\Phi(A) =\bigoplus \gamma_k(A)A_k.$$

Hence, there exist a normalized vector $z'\in M_0(R)=\Span\{y_i,y_j\}$ such that
$\langle Rz',\Phi(A)z'\rangle=0$. We can write it as 
$z'=c' y_i\oplus s'y_j$ for some  $c',s'\in\CC$, $|c'|^2+|s'|^2=1$, and then
\begin{align*}
 0&=\overline{\langle \Phi(R)z',\Phi(A)z'\rangle}=\overline{\langle \gamma(R)Rz',(\bigoplus\gamma_k(A)A_k)z'\rangle}\\
 &=\overline{\gamma(R)}\bigl(|c'|^2 \gamma_i(A) x_i^\ast A y_i-|s'|^2 \gamma_j(A) x_j^\ast A y_j\bigr)\\
 &=\overline{\gamma(R)}\bigl(|c'|^2 \gamma_i(A) \sigma_1(A_i)-|s'|^2 \gamma_j(A) \sigma_1(A_j)\bigr).  
\end{align*}

We see that 
$$\tfrac{\gamma_i(A)}{\gamma_j(A)}=\tfrac{|s'|^2}{|c'|^2}\cdot \tfrac{\sigma_1(A_j)}{\sigma_1(A_i)}>0.$$
In addition, if $\|A_i\|=\|A_j\|=\|A\|$, then, by Step~1, also $|\gamma_i(A)|=|\gamma_j(A)|$, and so $\|A_i\|=\|A_j\|=\|A\|$ implies $\gamma_i(A)=\gamma_j(A)$. Finally, choose any index $i_0$ with $A_{i_0}\neq0$,  (re)define $\gamma_i(A):=\gamma_{i_0}(A)$ whenever $A_i=0$ and then let $\gamma(A):=\frac{\gamma_1(A)}{|\gamma_1(A)|}$ and $P(A)=|\gamma_1(A)|\left(I_{n_1}\oplus\frac{\gamma_2(A)}{\gamma_1(A)} I_{n_2}\oplus\dots\oplus\frac{\gamma_l(A)}{\gamma_1(A)}I_{n_\ell}\right)>0$. 
\end{proof}

\section{Concluding remarks}
Let us show that if $C^*$-algebra is not simple, then not every BJ isomorphism on $\A$ is  an isometry multiplied by a scalar-valued function.
\begin{example}
Consider $\A=M_{n_1}(\CC)\oplus M_{n_2}(\CC)$ with $n_1,n_2\ge 2$ and let $\Phi\colon\A\to\A$ be a map which fixes all elements except those inside the set $I \oplus (0,1)I$, and let
$$\Phi(I\oplus r I) =I\oplus \gamma(r) I$$
for some bijection $\gamma\colon(0,1)\to(0,1)$.
Such $\Phi$ is bijective, and clearly strongly preserves BJ orthogonality except possibly if $A\perp B$ where  one among $A,B$ belong to $I\oplus (0,1)I$. 

\textbf{Case 1.} $A=I\oplus rI\in I\oplus(0,1)I$. 
Then $A$ achieves its norm only on its first summand, so $B=B_1\oplus B_2$ satisfies $A\perp B$ if and only if $I\perp B_1$. Then, $B_1\neq I$, and so $\Phi(B)=B$. It follows that $\Phi(A)=I\oplus \gamma(r)I \perp B=\Phi(B) $.

\textbf{Case 2.} $B=I\oplus r I\in I\oplus (0,1)I$ and decompose $A=A_1\oplus A_2$. Due to $A\perp B$ we clearly have that $A\not\in I\oplus (0,1)I$, so it is fixed by $\Phi$. 
Hence, to prove \begin{equation}\label{eq:counterexa}
A=\Phi(A)\perp\Phi(B)=I\oplus\gamma(r)I
\end{equation}
we only need to consider the case $\|A_1\|=\|A_2\|=\|A\|\neq0$. By Proposition ~\ref{prop:M_0(A)} there exists a  normalized vector $z=z_1\oplus z_2$, where $A$ attains its norm, such that 
\begin{equation}\label{eq:counter2}
 z_1^\ast A_1z_1+rz_2^\ast A_2 z_2=\langle Az,Bz\rangle=0.   
\end{equation}
 If $z_1=0$, or if $z_1\neq0$ but $z_2^\ast A_2z_2=0$, then $A_2\perp B_2=rI$, so also $A_2\perp \gamma(r) I$, which implies~\eqref{eq:counterexa}; likewise if $z_2=0$ (or if $z_1^\ast A_1z_1\neq0$). So can assume 
$$z_1^\ast A_1z_1 \neq 0\hbox{ and } z_2^\ast A_2z_2\neq0.$$
Notice that $A_1$ and $A_2$ then attain their norms on $z_1$ and $z_2$, respectively. As such $A$ attains its norm also on $z_c:=c z_1\oplus sz_2$ for every $c,s\in\CC$ with $|c|^2+|s|^2=1$. We can choose  $c$ and $s$ such that 
$$\tfrac{|s|^2}{|c|^2} \gamma(r)= r.$$ Then, by~\eqref{eq:counter2}, $\Phi(A)z_c= c A_1z_1\oplus s A_2z_2 $ and $\Phi(B)z_c=c z_1\oplus \gamma(r) s z_2$ are perpendicular vectors, so $\Phi(A)\perp \Phi(B)$. The same arguments can be applied on $\Phi^{-1}$ to show that $\Phi$ strongly preserves BJ orthogonality.
\end{example}

If $\A$ has abelian summand, then not every BJ isomorphism takes the form of Theorem \ref{thm:genera}.
\begin{example}
    Consider $\A=\CC\oplus\CC\cong\CC^2$ and let $\Phi:\A\to\A$ be a map which fixes all elements except those on the line $(1,ri),r\in\mathbb{R}$, and let
    $$\Phi((1,ri))=(1,-ri)=(1,-1)(1,ri).$$
    This $\Phi$ is bijective and strongly preserves $BJ$ orthogonality. However, it is not of the form from Theorem~\ref{thm:genera} since it multiplies some of the  arguments with a non-definite element $(1,-1)$.
\end{example}
\bibliographystyle{abbrv}
\bibliography{BJ_orthogonality}

\end{document}